\documentclass[11pt,reqno]{amsart}

\usepackage[T1]{fontenc}
\usepackage{lmodern}
\usepackage{microtype}
\usepackage[margin=1.12in]{geometry}
\usepackage{amsmath,amssymb,amsfonts,mathtools}
\usepackage{mathrsfs}
\usepackage{enumitem}
\usepackage{booktabs}
\usepackage{array}
\usepackage{tikz-cd}
\usepackage{xcolor}
\usepackage{aliascnt}
\usepackage{hyperref}
\usepackage[nameinlink,capitalise,noabbrev]{cleveref}
\usepackage{url}
\usepackage{longtable}
\usepackage{tikz}
\usepackage{float}
\usetikzlibrary{arrows.meta,shapes.geometric}
\hypersetup{
  colorlinks=true,
  linkcolor=blue!45!black,
  citecolor=green!35!black,
  urlcolor=blue!50!black,
}

\allowdisplaybreaks
\setlist[enumerate]{leftmargin=2.2em}
\setlist[itemize]{leftmargin=2.0em}

\newtheorem{theorem}{Theorem}[section]

\newaliascnt{proposition}{theorem}
\newtheorem{proposition}[proposition]{Proposition}
\aliascntresetthe{proposition}

\newaliascnt{lemma}{theorem}
\newtheorem{lemma}[lemma]{Lemma}
\aliascntresetthe{lemma}

\newaliascnt{corollary}{theorem}
\newtheorem{corollary}[corollary]{Corollary}
\aliascntresetthe{corollary}

\newaliascnt{remark}{theorem}
\newtheorem{remark}[remark]{Remark}
\aliascntresetthe{remark}

\newaliascnt{definition}{theorem}
\newtheorem{definition}[definition]{Definition}
\aliascntresetthe{definition}

\newaliascnt{example}{theorem}

\aliascntresetthe{example}

\newcommand{\Sp}{\mathrm{Sp}}
\newcommand{\Mod}{\mathrm{Mod}}
\newcommand{\Comod}{\mathrm{Comod}}
\newcommand{\GrAb}{\mathrm{GrAb}}
\newcommand{\Ab}{\mathrm{Ab}}
\newcommand{\Spec}{\operatorname{Spec}}

\newcommand{\Tor}{\operatorname{Tor}}
\newcommand{\ann}{\operatorname{ann}}
\newcommand{\res}{\operatorname{res}}
\newcommand{\im}{\operatorname{im}}
\newcommand{\MFG}{\mathcal M_{\mathrm{FG}}}
\newcommand{\bbT}{\mathbb T}
\newcommand{\bbZ}{\mathbb Z}

\newcommand{\tensor}{\otimes}
\newcommand{\lessdotp}{\mathrel{\lessdot_p}}
\newcommand{\Max}{\operatorname{Max}}

\DeclareMathOperator{\Hom}{Hom}

\DeclareMathOperator{\Ext}{Ext}

\title[Equivariant Landweber exactness]{An Equivariant Landweber Exact Functor Theorem for Abelian Compact Lie Groups}
\author{Yingxin Li}
\address{School of Mathematical Sciences, Nankai University, Tianjin, China}
\email{yingxinli@nankai.edu.cn}

\subjclass[2020]{Primary 55N22; Secondary 55N91, 55P91, 14L05, 57R85}
\keywords{equivariant Landweber exactness, equivariant formal group, equivariant Lazard ring, Euler class, exact pair of zero divisors}

\begin{document}

\begin{abstract}
    We prove an equivariant Landweber exact functor theorem for abelian compact Lie groups $G$. For a graded module $N$ over the $G$-equivariant Lazard ring $L_G$, we give necessary and sufficient algebraic conditions on $N$ for the functor
    \[
    X\longmapsto (MU_G)_*(X)\otimes_{L_G} N
    \]
    to define a homology theory on the category $\Sp^G$ of genuine $G$-spectra. Our criterion is obtained by varifying the Landweber exactness on each stratum indexed by a closed subgroup of $G$ and then gluing the local flatness data via Euler classes. As an application, we show that for any non-equivariant Landweber exact ring spectrum $E$, $MU_G\wedge_{MU} E$ is equivariantly Landweber exact, thereby proving a conjecture of Wisdom. Finally, we study the $G$-equivariant formal groups $\widehat E_G$ over $\hom(G^*,E)$ associated to elliptic curves $E/S$. We prove that its classifying map $\hom(G^*,E)\to\mathcal M_{\mathrm{FG}}^G$ is flat whenever $S\to\mathcal M_{\mathrm{ell}}$ is flat, and recover $\widehat E_G$ from Gepner-Meier's equivariant elliptic cohomology associated to any oriented spectral lifting of $E/S$.
\end{abstract}

\maketitle
\tableofcontents

\section{Introduction}
Given a graded formal group law over a commutative graded ring $R$, whose underlying formal group is classified by a morphism $\Spec R\to\MFG$, it is natural to ask whether it can be realized by a complex-oriented ring spectrum $E$ with the prescribed formal group law and natural isomorphisms
\[
E_*(X)\cong  MU_*(X)\tensor_{ MU_*}R.
\]
The classical Landweber exact functor theorem (LEFT) \cite{landweber1973associated,landweber1976homological} provides an explicit algebraic criterion ensuring that the functor $MU_*(-)\otimes_{MU_*} R$ defines a homology theory. A homology theory obtained in this way is called Landweber exact.

For an abelian compact Lie group $G$, Cole, Greenlees, and Kriz introduced $G$-equivariant formal group laws \cite{cole2000equivariant} and showed that every $G$-equivariant complex-oriented ring spectrum determines a $G$-formal group law. These notions provide a framework for equivariant chromatic homotopy theory, whose foundations---including the coefficient ring of $ MU_G$, the Balmer spectrum of finite $G$-spectra, and the moduli stack $\MFG^G$ of equivariant formal groups---have been studied in \cite{greenlees2001equivariant,cole2000equivariant,strickland2011multicurves,hanke2018equivariant,barthel2020balmer,hu2021equivariant,hausmann2022global,hausmann2023invariant,wisdom2024properties,behrens2024periodic,groenjes2026homology}.
In particular, Hausmann \cite{hausmann2022global} proved the equivariant analogue of Quillen's theorem, identifying the $G$-equivariant Lazard ring $L_G$ with the coefficient ring of $ MU_G$:
\[
L_G\xrightarrow{\;\cong\;}\pi_*^G MU_G.
\]
Then it natural to ask the equivariant version of the realization problem for a $G$-equivariant formal group law over some $L_G$-algebra $R$.

Let $N$ be a graded $L_G$-module. We say that $N$ is \emph{$G$-equivariant Landweber exact} if the functor
\[
h_N\colon \Sp^G\longrightarrow\GrAb, \qquad X\longmapsto  MU_*^G(X)\otimes_{L_G}N
\]
defines a homology theory. A $G$-equivariant formal group law $\mathbb G$ over a commutative graded ring $R$, classified by a graded ring homomorphism $L_G\to R$, is called \emph{Landweber exact} if $R$ is equivariantly Landweber
exact as an $L_G$-module. A complex-oriented $G$-ring spectrum $E$ is called \emph{$G$-equivariant Landweber exact} if there are natural isomorphisms
\[
E_*^G(X)\cong  MU_*^G(X)\otimes_{L_G}\pi_*^G E
\]
for all genuine $G$-spectra $X$. Examples of equivariantly Landweber exact spectra, including $KU_G$ and $BP_G$, are studied in \cite{wisdom2024properties}. However, there is currently no criterion to determine whether an equivariant formal group law arises from a equivariant complex-oriented ring spectrum. In this paper, we prove an equivariant Landweber exact functor theorem for abelian compact Lie groups $G$ that provides an algebraic criterion for equivariant Landweber exactness.

In the nonequivariant setting, the classical Landweber exact functor theorem states that a graded $MU_*$-module $N$ is Landweber exact if and only if the sequence $(p,v_1,v_2,\dots)$ is regular on $N$ for every prime $p$. More recently, Hausmann and Meier \cite{hausmann2023invariant} studied the flat Hopf algebroid $(L_G,\Gamma_G)$ classifying $A$-equivariant formal group laws and strict isomorphisms, identified its associated stack with the moduli stack $\MFG^A$, and constructed equivariant analogues of the classical $v_n$-elements, with chromatic height data encoded by height functions. However, these lifts do not directly yield an equivariant analogue of the classical regular-sequence criterion, as they are annihilated by powers of Euler classes; see \cite[Remark 6.7]{hausmann2023invariant}. Thus, the additional equivariant information cannot be encoded simply by arranging these lifts into a single global regular sequence.

Rather than seeking a single global regular sequence, we establish the flatness of an $L_G$-module $N$ over $\MFG^G$ by gluing its flatness over the Euler-localized substack $\MFG^{G/G}$ with the flatness of $N/e_\chi N$ over the closed substack $\MFG^{\ker(\chi)}$ for each nontrivial character $\chi\in G^*$. Here, $G^*=\hom(G,S^1)$ is the group of characters, and $e_\chi$ denotes the Euler class associated with $\chi$. A key ingredient in this gluing argument is the notion of an exact pair of zero divisors. Following Henriques and \c{S}ega \cite{bonacho2011free}, a pair $(a,b)$ of elements in a commutative ring $S$ is called an \emph{exact pair of zero divisors} if
\[
\operatorname{ann}_S(a)=bS \qquad\text{and}\qquad \operatorname{ann}_S(b)=aS.
\]
In our setting, we construct a system of such exact pairs on $L_G$ from suitable Euler classes and their companion elements. These pairs give rise to two-periodic free resolutions and, provided that they remain exact on the modules under consideration, yield $\Tor$-vanishing analogous to that associated with non-zero divisors. This vanishing plays a central role in the gluing argument.

\subsection{Statement of main results}
Let $G=\bbT^m\times A$ be an abelian compact Lie group, where $\bbT=S^1$ denotes the circle group and $A$ is a finite abelian group. Let $L_G$ denote the $G$-equivariant Lazard ring and $(L_G,\Gamma_G)$ the $G$-equivariant Lazard Hopf algebroid. For a graded $L_G$-module $N$ and each closed subgroup $H\leq G$, set
\[
N_H=L_H\tensor_{L_G}N,
\qquad
\Phi^H N=N_H[e_\chi^{-1}\mid 1\neq\chi\in H^*].
\]
For each prime $p$, we use the subscript $(-)_{(p)}$ to denote $p$-localization.

We first consider the case $m=0$, so that $G=A$ is finite abelian. For each pair of subgroups $K<H\leq A$ with $H/K\cong C_p$ for some prime $p$, choose an identification $H/K\cong C_p\subseteq\bbT$ and let $u\colon H\to\bbT$ be the resulting quotient character. Set
\[
e_{H/K}=e_u,
\qquad
d_{H/K}=\operatorname{Tr}_K^H(1).
\]
By extending $u$ to a non-torsion character $\widetilde u\colon\widetilde H\to\bbT$ of a suitable abelian compact Lie group $\widetilde H$, we obtain a presentation
\[
L_H\cong L_{\widetilde H}/(\widetilde e_{H/K}\widetilde d_{H/K}),
\]
where $\widetilde e_{H/K}$ and $\widetilde d_{H/K}$ are suitable lifts of $e_{H/K}$ and $d_{H/K}$, such that $\widetilde e_{H/K}\widetilde d_{H/K}=e_{p\widetilde u}$ is a nonzero divisor. Using this construction, we prove that $(e_{H/K},d_{H/K})$ is an exact pair of zero divisors on $L_H$. For each subgroup $H\leq A$, we assemble these exact pairs over all maximal subgroups $K<H$ and apply \Cref{lem:finite-gluing} to establish the flatness of $N_H$ over $L_H$ by gluing the flatness over $L_K$ for all maximal $K<H$, together with the flatness over $\Phi^H L$. These gluing data admit a concrete geometric interpretation; see \Cref{lem:maximal-euler-suffices,lem:closed-basechange,rmk:geometric meaning}.

By induction on the order of $A$, we obtain the following equivariant Landweber exact functor theorem for finite abelian groups.

\begin{theorem}[\Cref{thm:finite equivariant-LEFT}]\label{Introthm:finite}
Let $A$ be a finite abelian group, and let $N$ be a graded $L_{A}$-module.  The following are equivalent.
\begin{enumerate}[label=\textup{(\roman*)}]
\item The functors $N\tensor_{L_{A}} MU_*^A(-)$ form a homology theory.
\item $-\tensor_{L_{A}}N$ is exact on graded $(L_{A},\Gamma_{A})$-comodules.
\item For every prime $p$, every $K<H\leq A$ with $H/K\cong \bbZ/p$,
\[
\ker(e_{H/K}:N_{H,(p)}\to N_{H,(p)})=d_{H/K}N_{H,(p)}, \qquad \ker(d_{H/K}:N_{H,(p)}\to N_{H,(p)})=e_{H/K}N_{H,(p)},
\]
and, for every $H\leq A$, and every $n\geq0$,
\[
  v_{p,n}:\Phi^H N_H/I_{p,n}\Phi^H N \longrightarrow \Phi^H N_H/I_{p,n}\Phi^H N
\]
is injective.
\end{enumerate}
If $N=R$ is an evenly graded $L_{A}$-algebra, these conditions are also equivalent to flatness of $\Spec R\to\MFG^A$.
\end{theorem}

For a general abelian compact Lie group $G=\bbT^m\times A$, the closed substacks $\MFG^{\ker(\chi)}\subseteq\MFG^G$ associated with non-trivial characters $\chi\in G^*$ form an infinite family whenever $m>0$. We therefore extend the gluing argument to infinite families; see \Cref{lem:compact-infinite-gluing}. In this setting, we treat the torus and finite factors separately, using the fact that the Euler class $e_\tau$ is a non-zero divisor in $L_G$ for every nontrivial character $\tau\in(\bbT^m)^*$. Applying the extended gluing argument and proceeding by induction on the pair $(m,|A|)$ in lexicographic order, we obtain the following equivariant Landweber exact functor theorem for abelian compact Lie groups.

\begin{theorem}[Equivariant LEFT for abelian compact Lie groups, \Cref{thm:compact-LEFT}]\label{introthm:compact-LEFT}
Let $G\cong T\times A$, where $T=\bbT^m$ and $A$ is finite. For a graded $L_G$-module $N$, the following are equivalent:
\begin{enumerate}[label=\textup{(\roman*)},leftmargin=3em]
\item The functors $ N\tensor_{L_{G}} MU_*^G(-) $ form a homology theory.
\item The functor
\[
N\otimes_{L_G}- \colon \operatorname{Comod}_{(L_G,\Gamma_G)} \longrightarrow\operatorname{Mod}_R
\]
is exact.
\item $N$ satisfies the conditions as following
\begin{itemize}[leftmargin=3em]
  \item[$(\mathrm{RE})$] For any $r\leq m$ and any linearly independent $\{\tau_i\}_r \subset T^* \subset G^*$, the sequence $(e_{\tau_1},\cdots, e_{\tau_r})$ is a regular sequence over $N$.
  \item[$(\mathrm{EU})$] For every $p$ and every $K\lessdot_p H\leq A$, $(\pi^* e_{H/K},\pi^* d_{H/K})$ is an exact pair of zero divisors on $N_{T\oplus H,(p)}$ via the pullback along $\pi^*\colon T\oplus H \to H$.
  \item[$\mathrm{(LE)}$] For every prime $p$, every closed subgroup $B\leq G$ and every $n\geq0$, multiplication by $v_{p,n}$ is injective on $\Phi^B N_B/I_{p,n}\Phi^B N_B$.
\end{itemize}
\end{enumerate}
\end{theorem}

As a corollary, we obtain a criterion for $G$-equivariant Landweber exactness of a complex-oriented $G$-ring spectrum $E$.
\begin{corollary}[\Cref{cor:compact representation}]\label{introcor:spectra level representation}
  Let $E$ be a complex oriented $G$-spectrum $E$. Then 
  \[
  E_*(X) \cong (MU_G)_*(X)\otimes_{L_G} \pi_*^G E
  \]
  if and only if $\pi_*^G E$ is landweber exact, and for any $H\leq G$, $\pi_*^H E \cong L_H\otimes_{L_G} \pi_*^G E$.
\end{corollary}

We apply this criterion to show that the $G$-equivariant extension $E_G$ of a nonequivariant Landweber exact theory $E$ is itself equivariantly Landweber exact. This proves a conjecture of Wisdom \cite[Conjecture 2.18]{wisdom2024properties}.
\begin{proposition}[\Cref{prop:conjecture-in-Wisdom}]\label{introprop:conjecture-in-Wisdom}
  If $E$ is a Landweber exact spectrum, then $MU_A \wedge_{MU} E$ is an $G$-Landweber exact spectrum.
\end{proposition}

Finally, for any elliptic curve $E$ over $S$, following Strickland's idea in \cite[Definition 7.3]{strickland2011multicurves}, we can define a $G$-formal group $\widehat{E_G}$ over $S_G:=\hom(G^*,E)$ as the formal completion of $E\times_{S} S_G$ along the divisor given by the tautological evaluation map $G^*\times S_G \to E\times_S S_G$. Applying our criterion for Landweber exactness, we can show that if $E/S$ is represented by a flat map $S\to \mathcal M_{ell}$, then $\widehat{E_G}$ is flat over $\MFG^G$.

\begin{proposition}[\Cref{prop:elliptic formal group is exact}]\label{introprop:elliptic formal group is exact}
  For any elliptic $E/S$ represented by a flat map $S\to \mathcal{M}_{ell}$, the $G$-equivariant formal group 
  \[
  \phi_G: S_G \to \MFG^G
  \]
  is flat.
\end{proposition}

Moreover, following ideas of Lurie in \cite[Proposition~3.1, Theorem~3.2]{lurie2009survey}, for an oriented elliptic curve $\mathsf{E}/\mathsf{M}$ over a non-connective spectral Deligne-Mumford stack $\mathsf{M}$, Gepner-Meier \cite[Construction~6.4]{gepner2023equivariant} construct an $G$-equivariant elliptic cohomology theory on $G$-spaces as a colimit-preserving functor
\[
\mathcal Ell_{G}: \mathrm{Space}^G \to \mathrm{QCoh}(\mathsf{ELL}(\mathbf BG))^{op},
\]
which takes values in the category of quasi-coherent sheaves over $\mathsf{ELL}(\mathbf BG)=\hom_{\mathrm{Ab}(\mathrm{Sh}(\mathsf M))}(G^*, \mathsf E)$. Let $\mathsf S=(S,\mathcal O_{\mathsf S})$ be a spectral scheme and let $\mathsf S_0=(S,\pi_0\mathcal O_{\mathsf S})$ be the underlying scheme of $\mathsf S$. For any oriented spectral elliptic curve $\mathsf E/\mathsf S$ with underlying classical elliptic curve $E/S=\mathsf E_0/\mathsf S_0$, together with a computation of $\mathcal Ell_G(\mathbb CP^V)$ in \Cref{lem:complex projective space CPV}, we constructed a $G$-formal group from $\mathcal Ell_G$ which is isomorphic to $\widehat{E_G}/S_G$. 

\begin{proposition}[\Cref{prop:compactible of elliptic formal groups}]\label{introprop:formal groups of elliptic cohomology}
  For such an elliptic curve $\mathsf E$ over $\mathsf S$,  
  $$\operatorname{colim}_V \big(\Spec_{\mathsf{ELL}(\mathbf BG)}(\mathcal{E}ll_{G}(\mathbb CP^V))\big)_0$$ 
  gives the $G$-equivariant formal group $\widehat{E_G}/S_G$, where the colimit takes over all finite dimensional $G$-representations $V$. Here $\Spec_{\mathsf E_G}$ is the inverse of $\mathrm{Aff}_{\mathsf E_G}^{op}\to \mathrm{CAlg}(\mathrm{QCoh}(\mathsf E_G))$ defined in \cite[Lemma 8.1]{gepner2023equivariant}. 
\end{proposition}
In particular, if $E/S$ is represented by an \'{e}tale morphism $S\to \mathcal M_{ell}$.
Let $\mathsf S=(S, \mathcal O^{top}|_{S})$, and let $\mathsf E \to \mathsf S$ be the oriented elliptic curve determined by $$\mathsf S\to \mathsf M_{ell}^{or},$$ where $\mathsf M_{ell}^{or}$ is the non-connective spectral Deligne-Mumford stack of oriented elliptic curves \cite[Proposition~7.2.10]{lurie2018elliptic}, and $\mathcal O^{top}$ is the structure sheaf of $\mathsf M_{ell}^{or}$. If $S_G$ is affine and $\widehat{E_G}/S_G$ admits a coordinate, then the equivariant elliptic cohomology theory associated to this $\mathsf E/\mathsf S$ is complex oriented and thus Landweber exact.

In general, our results provide a method for constructing $G$-equivariant Landweber exact spectra from $G$-equivariant formal group laws satisfying our criterion. We hope that these constructions will be useful in the study of equivariant elliptic cohomology, and more broadly, in the development of a derived theory of equivariant formal groups.

    \subsection{Outline}
    In \Cref{sec:background}, we recall the necessary background on equivariant formal groups and the moduli stack $\MFG^G$. In \Cref{sec:finite}, we first establish the equivalence between the Landweber exactness of a spectrum and algebraic Landweber exactness for certain evenly graded $G$-ring spectra $E$. We then prove the finite gluing lemma (\Cref{lem:finite-gluing}) and verify that $(e_{H/K},d_{H/K})$ is an exact pair on $L_H$ in the finite abelian case. Then we prove \Cref{Introthm:finite}. In \Cref{sec:compact}, we extend the gluing lemma to infinite families (\Cref{lem:compact-infinite-gluing}) and use it to prove \Cref{introthm:compact-LEFT} for abelian compact Lie groups. Then we deduce \Cref{introcor:spectra level representation} and \Cref{introprop:conjecture-in-Wisdom} as consequences. In \Cref{sec:equivariant elliptic}, we study the equivariant formal group from elliptic curves and elliptic cohomology theories and prove \Cref{introprop:elliptic formal group is exact,introprop:formal groups of elliptic cohomology}.

    \subsection*{Acknowledgements}
    I would like to thank my advisors, Xiangjun Wang and Xu-an Zhao, for their guidance throughout this project. I would especially like to thank Markus Hausmann for helpful discussions and suggestions on the Landweber exactness of elliptic cohomology. I also thank Noah Wisdom and Lucas Piessevaux for helpful discussions.

\section{Background of equivariant formal groups}\label{sec:background}
In this section, we briefly review some background on equivariant formal groups. We refer the reader to \cite{strickland2011multicurves,hausmann2023invariant} for a more detailed discussion. 

Let $G$ be an abelian compact Lie group, and let $G^*=\hom(G,S^1)$ be the character group. Let $S$ be a quasi-compact scheme. A $G$-equivariant formal group over $S$ consists of a commutative group object $X$ in the category of formal $S$-schemes together with a group homomorphism $$\phi:S\times G^*\to X$$ such that $X$ is (fpqc locally) the formal neighborhood of the divisor $[\phi(G^*)]$. That is, 
    \begin{itemize}
        \item[\quad (1)] Let $\epsilon\in G^*$ be the trivial representation of $G$. For the composite $\phi_\epsilon:S\overset{id_S\times \epsilon}{\longrightarrow}S\times G^*\overset{\phi}{\longrightarrow}X$, the ideal sheaf $\mathcal I_\epsilon=\mathcal O_X(-[\phi_\epsilon])$ is fpqc locally on $S$ a free $\mathcal O_X$ module of rank $1$.
    \item[\quad (2)] Fpqc locally on $S$, the topology on $R=\Gamma(\mathcal O_X)$ is generated by finite product of the ideals $I_\alpha=\ker \phi_\alpha$, i.e. $X=\operatorname{Spf}(R)$ where $R=\lim_V R/I_V$, $I_V=\prod_{i} I_{\alpha_i}$ for $V=\oplus_{i=1}^r \alpha_i$.
    \end{itemize}
  If there is an isomorphism $\mathcal I_\epsilon\cong \mathcal O_X$, then we say that $\phi$ admits a coordinate, and it gives a $G$-equivariant formal group law. In studying Landweber exactness, it suffices to consider coordinatizable formal groups over affine schemes. We therefore assume that $S=\Spec R$ for some commutative ring $R$. Let $y_\epsilon\in R=\mathcal O_X$ be a function on $X$ whose vanishing locus is the divisor $[\phi(\epsilon)]$. For each $\tau\in G^*$, define the associated Euler class by
\[
e_\tau=\phi_\tau^*(y_\epsilon),
\]
where $\phi_\tau=\phi(\tau)$.  We write
\[
L_G=\pi_*^G MU_G,
\qquad
\Gamma_G=\pi_*^G(MU_G\wedge MU_G).
\]
For each prime $p$, the subscript $(p)$ denotes $p$-localization, that is, tensoring over $\bbZ$ with $\bbZ_{(p)}$. For example,
\[
L_{G,(p)}=L_G\tensor_{\bbZ}\bbZ_{(p)},
\qquad
\Gamma_{G,(p)}=\Gamma_G\tensor_{\bbZ}\bbZ_{(p)}.
\]
Both $L_G$ and $\Gamma_G$ are concentrated in even degrees. Unless otherwise specified, graded rings are understood to be graded-commutative, and flatness refers to flatness of the underlying ungraded module.

The equivariant Lazard ring $L_G$ carries the universal $G$-equivariant formal group law. Moreover, Hausmann and Meier \cite[Proposition~3.2 and Corollary~3.7]{hausmann2023invariant} identify $(L_G,\Gamma_G)$ with the graded Hopf algebroid classifying $G$-equivariant formal group laws and strict isomorphisms. In particular, $\Gamma_G$ is flat over $L_G$ via either unit map. We denote the left and right unit maps of a Hopf algebroid by $\eta_L$ and $\eta_R$, respectively. Let $\MFG^G$ denote the moduli stack of $G$-equivariant formal groups, which is the stack associated with the flat Hopf algebroid $(L_G,\Gamma_G)$. 

\begin{definition}
Let $\alpha\colon H\to G$ be a continuous homomorphism of abelian compact Lie groups, and let $\mathbb G=(H^*\xrightarrow{\phi}X)$ be an $H$-equivariant formal group. The \emph{corestriction} $\alpha_*\mathbb G$ is the $G$-equivariant formal group obtained by completing $X$ along the image of the composite
\[
G^*\xrightarrow{\alpha^*}H^*\xrightarrow{\phi}X.
\]
\end{definition}

This construction induces a homomorphism of equivariant Lazard rings
\[
\alpha^*\colon L_G\longrightarrow L_H.
\]
If $\alpha$ is the inclusion of a closed subgroup, then $\alpha^*$ is surjective, with kernel generated by the Euler classes $e_V$ for all $V\in\ker(G^*\to H^*)$. In particular, if $H=\ker(V)$ for a character $V\colon G\to\bbT$, then there is an exact sequence
\[
L_G\xrightarrow{\,e_V\,}L_G\longrightarrow L_H\longrightarrow 0.
\]
If $V$ is non-torsion, then $e_V$ is a non-zero divisor in $L_G$.

On the other hand, when $q:G\to G/H$ is surjective, we can construct $G$-formal groups of pushout type from $G/H$-equivariant formal groups.
\begin{definition}
    Let $q:G\to G/H$ be a surjective group homomorphism and let $\mathbb G=((G/H)^*\to X)$ be a $(G/H)$-equivariant formal group. We define the coinduction $\alpha^* G$ to be the $G^*$ equivariant group $$(G^*\to X\times_{(G/H)^*}G^*=\sqcup_{\alpha\in G^*/(G/H)^*} X).$$ 
\end{definition}

\begin{proposition}{\cite[Proposition 3.11]{hausmann2023invariant}}\label{prop:substacks}
    Let $H\leq G$ be a closed subgroup of an abelian compact Lie group, and write $\alpha\colon H\hookrightarrow G$ and $q\colon G\to G/H$ for the inclusion and quotient homomorphisms, respectively. 
    \begin{enumerate}[label=\textup{(\roman*)}]
      \item $\alpha_*\colon\MFG^H\to\MFG^G$ is a closed immersion identifying $\MFG^H$ with the common vanishing locus of the Euler classes $e_V$ for all $V\in\ker(G^*\to H^*)$.
      \item $q^*\colon\MFG^{G/H}\to\MFG^G$ identifying $\MFG^{G/H}$ with the common non-vanishing locus of the Euler classes $e_V$ for all $V\notin \im((G/H)^*\to G^*)$.
    \end{enumerate}
\end{proposition}

\section{Landweber exactness for finite abelian groups}\label{sec:finite}

Let $(B,\Gamma)$ be a flat Hopf algebroid and let $f\colon B\to R$ be a ring homomorphism. Hovey-Strickland \cite[Lemma 2.2]{hovey2005comodules} show that the functor
\[
R\otimes_B -\colon \Comod_{B,\Gamma}\longrightarrow \Mod_R
\]
is exact if and only if the map $f\otimes \eta_R$ endows $R\otimes_B\Gamma$ with the structure of a flat $B$-algebra. We say that$R$ is algebraically Landweber exact over $(B,\Gamma)$ if it satisfies these equivalent conditions.

Now let $E$ be an $G$-ring spectrum and let $R$ be a $\pi_*^G E$-module. If $R$ is algebraically Landweber exact over the Hopf algebroid
\[
\bigl(\pi_*^G E,\pi_*^G(E\wedge E)\bigr),
\]
then the functor
\[
X\longmapsto
R\otimes_{\pi_*^G E}E_*^G(X)
\]
defines a homology theory on $\Sp^G$. For suitable flat $G$-ring spectra $E$, we first prove the converse, which is a generalization of \cite{rudyak1986exactness}: if this functor defines a homology theory on $\Sp^G$, then $R$ is algebraically Landweber exact over $\bigl(\pi_*^G E,\pi_*^G(E\wedge E)\bigr)$.

\begin{proposition}\label{prop:homology-implies-flat}
Let $E$ be an associative genuine homotopy commutative $G$-ring spectrum. Put
\[
  L=\pi_*^GE, \qquad \Gamma=\pi_*^G(E\wedge E),
\]
and suppose $L$ is even and $\Gamma$ is flat over $L$ through $\eta_R$. Let $N$ be a graded $L$-module.  If
\[
  h_n^N(X)=\bigl(N\tensor_LE_*^G(X)\bigr)_n
\]
forms a homology theory on $\Sp^G$, then $N\tensor_{L,\eta_L}\Gamma$ is flat over $L$ through $\eta_R$.
\end{proposition}

\begin{proof}
Let $a:F\to F'$ be a degree-zero map between finite generated graded free $L$-modules and let $K=\ker(a)$.  Choose finite free left $E$-modules $\mathbb F_E = \bigvee_i \Sigma^{r_i} E$ and $\mathbb F'_E =\bigvee_j \Sigma^{s_j} E$ realizing the grading shifts and realize $a$ by an $E$-module map
\[
 \mu_a:\mathbb F_E\longrightarrow\mathbb F'_E.
\]
Indeed, for integer shifts $r,s$ the free-module adjunction gives $[\Sigma^rE,\Sigma^sE]_{\Mod_E(\Sp^G)}\cong\pi_{r-s}^GE$, each $(i,j)$-entry $a_{i,j}:\Sigma^{r_i} E \to \Sigma^{s_j} E$ can be realized independently.  Let $C_a$ be its cofiber, regarded also as a genuine $G$-spectrum.  Applying $E_*^G(-)$ gives
\[
 E_{*+1}^G(C_a)\longrightarrow  \Gamma\tensor_{L,\eta_R} F \xrightarrow{1\tensor a} \Gamma\tensor_{L,\eta_R}F'.
\]
Because $\Gamma$ is $\eta_R$-flat, the kernel of the second map is $\Gamma\tensor_{L,\eta_R}K$; hence the boundary factors through a surjection
\[
 E_{*+1}^G(C_a)\twoheadrightarrow \Gamma\tensor_{L,\eta_R}K.
\]
All maps are $L$-linear through $\eta_L$.

Apply the assumed homology theory to the same cofiber sequence.  The free terms become
\[
 h_*^N(\mathbb F_E)=N\tensor_{L,\eta_L}\Gamma\tensor_{L,\eta_R}F,
 \qquad
 h_*^N(\mathbb F'_E)=N\tensor_{L,\eta_L}\Gamma\tensor_{L,\eta_R}F'.
\]
Tensoring the preceding surjection with $N$ remains surjective.  Exactness of the homology sequence therefore identifies
\[
 \ker(1\tensor a:N\tensor_{L,\eta_L}\Gamma\tensor_LF\to N\tensor_{L,\eta_L}\Gamma\tensor_LF')
 =\im(N\tensor_{L,\eta_L}\Gamma\tensor_LK).
\]
Thus, for every degree-zero homomorphism $a\colon F\to F'$ between finitely generated graded free $L$-modules, with $K=\ker(a)$, the sequence
\[
N\tensor_{L,\eta_L}\Gamma\tensor_L K \longrightarrow N\tensor_{L,\eta_L}\Gamma\tensor_L F \xrightarrow{\,1\tensor a\,} N\tensor_{L,\eta_L}\Gamma\tensor_L F'
\]
is exact at the middle term. The result now follows from the graded analogue of the equational criterion for flatness (cf.~\cite[Lemma 10.39.11]{stacks-project}). Here, the evenness of $L$ ensures that its underlying ring is commutative and that no signs arise in the argument.
\end{proof}

\begin{proposition}\label{thm:first equivalent conditions}
Consider the $G$-equivariant complex cobordism spectrum $MU_G$. For a graded $L_G$-module $N$, the following are equivalent.
\begin{enumerate}
\item $\bigl(N\tensor_L MU_*^G(-)\bigr)_n$ is a homology theory.
\item $N\tensor_{L_G,\eta_L}\Gamma_G$ is flat over $L_G$ through $\eta_R$.
\item $-\tensor_LN$ is exact on graded $(L_G,\Gamma_G)$-comodules.
\end{enumerate}
\end{proposition}

\begin{proof}
The implication $(1)\Rightarrow(2)$ is
\cref{prop:homology-implies-flat}, and $(2)\Leftrightarrow(3)$ is
\cite[Lemma 2.2]{hovey2005comodules}.  If (3) holds, apply it to the comodule-valued long exact sequence of a cofiber sequence and then take homogeneous components. Suspension and coproduct axioms are inherited from $MU_*^G(-)$.
\end{proof}

Now it suffices to study the flatness of $N\tensor_{L_G,\eta_L}\Gamma_G$ over $L_G$. In the remainder of this section, let $A$ be a finite abelian group. We firstly consider the flatness for finite abelian groups, for which the geometric interpretation of the gluing argument is more transparent. In the next section, we extend the argument to general abelian compact Lie groups. Except for \Cref{lem:finite-gluing,lem:maximal-euler-suffices}, the results of this section remain valid for abelian compact Lie groups.

\subsection{The gluing lemma}
Let $B$ be a graded-commutative ring and $M$ a graded $B$-module. A pair of
homogeneous elements $(x,y)$ is an \emph{exact pair of zero divisors on $M$}
if
\[
 \ker(x:M\to M)=yM,
 \qquad
 \ker(y:M\to M)=xM.
\]
Equivalently, the complex 
\[
\cdots \longrightarrow M \overset{x}{\longrightarrow} M \overset{y}{\longrightarrow}  M \overset{x}{\longrightarrow} M \to M \overset{y}{\longrightarrow} M \longrightarrow \cdots
\]
is exact.

\begin{lemma}\label{lem:tor-exact-pair}
If $(x,y)$ is an exact pair on both $B$ and $M$, then
\[
 \Tor_j^B(B/xB,M)=0
 \qquad(j>0).
\]
\end{lemma}

\begin{proof}
The exact pair on $B$ gives a free resolution
\[
  \cdots \overset{y}{\longrightarrow} B \overset{x}{\longrightarrow} B
  \overset{y}{\longrightarrow} B \overset{x}{\longrightarrow}B\longrightarrow B/xB\longrightarrow 0.
\]
After tensoring with $M$, the resulting complex is exact because $(x,y)$ is an
exact pair on $M$. Its positive homology groups are the asserted Tor groups.
\end{proof}

\begin{lemma}\label{lem:one-gluing}
Let $B$ be a commutative ring, $x\in B$, and $M$ a $B$-module. Assume
\begin{enumerate}
\item $\Tor_j^B(B/xB,M)=0$ for $j=1,2$;
\item $M/xM$ is flat over $B/xB$;
\item $M[x^{-1}]$ is flat over $B[x^{-1}]$.
\end{enumerate}
Then $M$ is flat over $B$.
\end{lemma}

\begin{proof}
  This is precisely \cite[Lemma 10.99.17]{stacks-project} in the special case where the ideal is generated by a single element.
\end{proof}

\begin{corollary}[One step gluing]\label{cor:pair-gluing}
If $(x,y)$ is an exact pair on $B$ and on $M$, if $M/xM$ is flat over
$B/xB$, and if $M[x^{-1}]$ is flat over $B[x^{-1}]$, then $M$ is flat over
$B$.
\end{corollary}

\begin{proof}
Since $(x,y)$ is an exact pair on $B$ and on $M$, \cref{lem:tor-exact-pair} implies that $\Tor_j^B(B/xB,M)=0$ for every $j>0$. Then \Cref{lem:one-gluing} leads to the statement.
\end{proof}

\begin{lemma}[Finite Euler gluing]\label{lem:finite-gluing}
Let $(x_i,y_i)$, $1\leq i\leq m$, be exact pairs on a commutative ring $B$ and
on a $B$-module $M$. Suppose
\[
 M/x_iM\text{ is flat over }B/x_iB
 \quad\text{for every }i,
\]
and
\[
 M[(x_1\cdots x_m)^{-1}]
 \text{ is flat over }B[(x_1\cdots x_m)^{-1}].
\]
Then $M$ is flat over $B$.
\end{lemma}
\begin{proof}
We prove the proposition by induction on $m$. The case $m=1$ is exactly
\cref{cor:pair-gluing}. Suppose $m>1$ and localize at $x_1$. For
$i\geq2$, $(x_i,y_i)$ is an exact pair of zero divisors on $M[x_1^{-1}]$, and
\[
 \frac{M[x_1^{-1}]}{x_iM[x_1^{-1}]}
 \cong (M/x_iM)[x_1^{-1}]
\]
is flat over $(B/x_iB)[x_1^{-1}]$. Replace $M$ with $M[x_1^{-1}]$, the induction hypothesis shows that $M[x_1^{-1}]$ is flat over $B[x_1^{-1}]$. Then applying \cref{cor:pair-gluing} to $(x_1,y_1)$ shows that $M$ is flat over $B$.
\end{proof}

\subsection{Exact pair of zero divisors in $L_A$}
For each prime $p$, choose a standard $p$-typical sequence
\[
  v_{p,0}=p,v_{p,1},v_{p,2},\ldots
\]
in $L_{(p)}$ and put
\[
 I_{p,n}=(v_{p,0},\ldots,v_{p,n-1})\subseteq L_{(p)},
 \qquad I_{p,0}=(0).
\]
Let $\tau:\bbT\to\bbT$ be the identity character.  The Euler
class $e_{\tau^p}$ is divisible by $e_\tau$ since it restricts to $0$ at the trivial group; write
\[
  e_{\tau^p}=\psi_p \cdot e_\tau, \qquad \res_1^{\bbT}(\psi_p)=p.
\]
Here, $\psi_p$ is precisely the element defined in \cite[Proposition 5.46]{hausmann2022global}.
After $p$-localization, Hausmann--Meier construct for every $n\geq0$ a unique
class
\[
 \psi_p^{(n)}\in L_{\bbT,(p)}/I_{p,n}
\]
such that
\[
  e_{\tau^{p}}=\psi_p^{(n)}e_\tau^{p^n},  \qquad  \res_1^{\bbT}(\psi_p^{(n)})=v_{p,n}.
\]
For $n=0$, this is the $p$-localization of $\psi_p$\cite[Remark 5.13]{hausmann2023invariant}.  The construction and restriction calculation occur immediately before \cite[Proposition 5.10]{hausmann2023invariant}.

Write
\[
  K\mathrel{\lessdot_p}H
\]
when $K<H$ and $H/K\cong C_p$.  Choose a quotient $q_{H/K}:H\twoheadrightarrow C_p$ and a faithful character $\bar\tau$ of $C_p$.  Define
\[
  e_{H/K}=e_{q_{H/K}^*\bar\tau}\in L_H
\]
and, after $p$-localization, define
\[
 d_{H/K}^{(p,n)}
 =q_{H/K}^*\bigl(\bar\psi_p^{(n)}\bigr)
 \in L_{H,(p)}/I_{p,n},
\]
where $\bar\psi_p^{(n)}$ is the restriction of $\psi_p^{(n)}$ to $C_p$.
At height zero we use the class
\[
  d_{H/K}=q_{H/K}^*(\bar\psi_p)\in L_H,
\]
where $\bar\psi_p$ denotes the restriction of $\psi_p$ to $C_p$; its $p$-localization is $d_{H/K}^{(p,0)}$. Note that
\[
e_{H/K}=e_{\chi_{H/K}},\qquad
d_{H/K}=\operatorname{tr}_K^H(1),
\]
where $\chi_{H/K}:H\to\mathbb T$ is the resulting quotient character.

\begin{lemma}\label{lem:adapted-group}
Let $H$ be a finite abelian group and let $K\mathrel{\lessdot_p}H$.  There
exist a rank-one abelian compact Lie group $\widetilde H$, an inclusion
$H\hookrightarrow\widetilde H$, and a surjective non-torsion character
$u:\widetilde H\to\bbT$ such that
\[
 \ker(u)=K,
 \qquad
 \ker(pu)=H,
\]
and $u|_H$ is the quotient character $H\to H/K\cong C_p$.
\end{lemma}

\begin{proof}
Pontryagin duality turns $K\lessdot_p H$ into an extension
\[
 0\longrightarrow C_p\longrightarrow H^*\longrightarrow K^*
 \longrightarrow0.
\]
Apply $\Hom_{\bbZ}(K^*,-)$ to $0\to\bbZ\xrightarrow{p}\bbZ\to C_p\to0$, here $\bbZ$ is the Pontryagin dual of $\bbT$. A finite abelian group has
projective dimension one over $\bbZ$, so
$\Ext^2_{\bbZ}(K^*,\bbZ)=0$ and the extension lifts to
\[
 0\longrightarrow\bbZ\langle u\rangle\longrightarrow X
 \longrightarrow K^*\longrightarrow0.
\]
Here $\bbZ\langle u\rangle$ is the infinite cyclic group generated by $u:\bbT \to \bbT$.
Its pushout along $\bbZ\to C_p$ is $H^*$.  Putting $\widetilde H=X^*$ and dualizing yields the desired construction.
\end{proof}

\begin{proposition}\label{prop:universal-edge}
Let $K\lessdot_p H$ and $n\geq0$.  In
$L_{H,(p)}/I_{p,n}L_{H,(p)}$ one has
\[
  \ann(e_{H/K}^{p^n})=(d_{H/K}^{(p,n)}), \qquad \ann(d_{H/K}^{(p,n)})=(e_{H/K}^{p^n}).
\]
At height zero there is also an integral assertion in $L_H$:
\[
 \ann_{L_H}(e_{H/K})=(d_{H/K}),
 \qquad
 \ann_{L_H}(d_{H/K})=(e_{H/K}).
\]
\end{proposition}
\begin{proof}
Choose $(\widetilde H,u)$ as in \cref{lem:adapted-group} and set
\[
 D_n=L_{\widetilde H,(p)}/I_{p,n}L_{\widetilde H,(p)}.
\]
Pulling $e_{\tau^p}$ back along $u$ gives
\[
 e_{pu}=u^*(\psi_p^{(n)})e_u^{p^n}
 \quad\text{in }D_n.
\]
The character $pu$ is non-torsion, so $e_{pu}$ is a non-zero divisor in $D_n$, and $L_{H,(p)}/I_{p,n}\cong D_n/(e_{pu})$. 
Then the images of $u^*(\psi_p^{(n)})$ and $e_u^{p^n}$ in $D_n/u^*(\psi_p^{(n)})e_u^{p^n}$ satisfy
\[
 \ann(\bar u^*(\psi_p^{(n)}))=(\bar e_u^{p^n}),
 \qquad
 \ann(\bar e_u^{p^n})=(\bar u^*(\psi_p^{(n)})),
\]
The two factors restrict to $e_{H/K}^{p^n}$ and $d_{H/K}^{(p,n)}$. 

For height zero, use the factorization
$e_{pu}=u^*(\psi_p)e_u$. Here $e_{pu}$ is a non-zero divisor and $L_H\cong L_{\widetilde H}/e_{pu}$. The same argument in $L_{\widetilde H}$ therefore proves both annihilator equalities directly.
\end{proof}

As a corollary, the ideal $(d_{H/K})\subset L_H$ and $(d_{H/K}^{(p,n)})\subset L_{H,(p)}/I_{p,n}$ do not depend on the chosen generator of $(H/K)^*$ since they are the annihilators of the corresponding powers of $e_{H/K}$.

\begin{lemma}\label{prop:d-invariant}
For every $K\lessdot_p H$ and $n\geq0$, the ideals $(e_{H/K}^{p^n})$ and $(d_{H/K}^{(p,n)})$ are invariant in $L_{H,(p)}/I_{p,n}$.  When $n=0$, $(e_{H/K})$ and $(d_{H/K})$ is invariant in $L_H$.
\end{lemma}
\begin{proof}
  The principal ideal generated by an Euler class is invariant since their vanishing locus gives rise to a closed substack of $\MFG^H$. Consider Hopf algebroid $$(B,C)=(L_{H,(p)}/I_{p,n},\Gamma_{H,(p)}/\eta_L(I_{p,n})\Gamma_{H,(p)})$$ and put $x=e_{H/K}^{p^n}$ and $y=d_{H/K}^{(p,n)}$.  By \cref{prop:universal-edge},
  \[
  \ann_C(\eta_L(x))=\eta_L(y)C, \qquad \ann_C(\eta_R(x))=\eta_R(y)C.
  \]
  The ideals $\eta_L(x)C$ and $\eta_R(x)C$ agree, hence so do their annihilators.  This proves invariance of $(d_{H/K}^{(p,n)})$.  At height zero the same argument applies directly to $(L_H,\Gamma_H)$, using the $n=0$ part of \cref{prop:universal-edge}.
\end{proof}

\subsection{The gluing data}
For a subgroup $H\leq A$, restriction gives a surjection $L_A\to L_H$.  If $H\hookrightarrow A$, its kernel is generated by Euler classes corresponding to a generating set of $\ker(A^*\to H^*)$ \cite[Lemma 2.21]{hausmann2023invariant}.  For a graded $L_A$-module $N$, put
\[
  N_H=L_H\tensor_{L_A}N.
\]
Write
\[
  \Phi^HL =L_H[e_\chi^{-1}\mid1\neq\chi\in H^*].
\]
Hausmann-Meier\cite[Proposition 2.25]{hausmann2023invariant} calculate
\[
  \Phi^H L \cong  L[(b_0^\chi)^{\pm1},b_i^\chi\mid i>0,\ 1\neq\chi\in H^*]
\]
Thus $\Phi^H L$ is faithfully flat over $L$.  Set
\[
  \Phi^H N=\Phi^H L \otimes_{L_H} N_H.
\]
Let
\[
  \Max(H)=\{K < H\mid K\text{ is maximal}\}.
\]
Every $H/K$ is cyclic of prime order.  Choose one quotient Euler class for each $K\in\Max(H)$ and set
\[
  \Delta_H=\prod_{K\in\Max(H)}e_{H/K}, \qquad \Delta_1=1.
\]

\begin{lemma}\label{lem:maximal-euler-suffices}
For every finite abelian group $H$,
\[
 L_H[\Delta_H^{-1}]\cong \Phi^HL.
\]
Consequently, for every $L_H$-module $M$, $M[\Delta_H^{-1}]= \Phi^H L \tensor_{L_H} M$.
\end{lemma}

\begin{proof}
Let $\chi\in H^*$ be nontrivial of order $m$.  Choose a prime $q\mid m$ and put $\lambda=(m/q)\chi$, which has order $q$.  Then $K=\ker(\lambda)$ is maximal.  The characters $\lambda$ and the chosen quotient character $q_{H/K}^*\bar\tau$ for $H/K$ generate the same cyclic subgroup of $H^*$, so their Euler classes generate the same principal ideal $\ker(L_H\to L_K)$.  Hence $e_\lambda$ becomes invertible after inverting $\Delta_H$.

Put $P=\ker(\chi)$.  The character $\chi$ generates $\ker(H^*\to P^*)$, and
\[
  \ker(L_H\xrightarrow{\res_P^H}L_P)=(e_\chi).
\]
Since $\lambda|_P$ is trivial, $\res_P^H(e_\lambda)=0$.  It follows that $e_\lambda\in(e_\chi)$; thus $e_\lambda=c_\chi e_\chi$ for some $c_\chi\in L_H$.  The already inverted element $e_\lambda$ makes both factors invertible, and in particular $e_\chi$ is invertible.  Every nontrivial Euler class is therefore invertible in $L_H[\Delta_H^{-1}]$.  The reverse inclusion is immediate, and the module assertion follows by tensoring.
\end{proof}

\begin{lemma}\label{lem:closed-basechange}
For subgroups $K\leq H$ there is a canonical isomorphism
\[
  \Gamma_K\cong L_K\tensor_{L_H,\eta_L}\Gamma_H \tensor_{L_H,\eta_R}L_K.
\]
In particular, if $K=\ker V$ for a character $V\in H^*$, $\Gamma_K\cong \Gamma_H/e_V$. 
\end{lemma}
\begin{proof}
Note that $\MFG^K\to\MFG^H$ is the closed substack given by the common zero locus of Euler classes for $V\in \ker(H^*\to K^*)$. The pullback diagram
\[
\begin{tikzcd}
    \Spec L_K \arrow[r]\arrow[d] & \Spec L_H \arrow[d]  \\
    \MFG^K \arrow[r] &\MFG^H  
\end{tikzcd}
\]
leads to isomorphisms of affine schemes
\[ 
  \Spec \Gamma_K \cong \Spec L_K \times_{\MFG^K} \Spec L_K \cong \Spec L_K \times_{\MFG^H} \Spec L_K \cong \Spec L_K\tensor_{L_H}\Gamma_H \tensor_{L_H}L_K,
\]
This is equivalent to the isomorphism $\Gamma_K\cong L_K\tensor_{L_H,\eta_L}\Gamma_H \tensor_{L_H,\eta_R}L_K$. When $K\cong \ker V$, the invariance of the ideal $(e_V)\subseteq L_H$ implies that $\Gamma_K\cong\Gamma_H/e_V$.
\end{proof}

\begin{remark}\label{rmk:geometric meaning}
  From the viewpoint of moduli stacks, \Cref{lem:closed-basechange} follows from the affineness of the closed immersion $\MFG^K\hookrightarrow\MFG^H$. For finite abelian $H$ and a maximal subgroup $K<H$, the a quotient character $u:H\to H/K$ generates $\ker(H^*\to K^*)$, so the vanishing locus of $e_u$ in $\MFG^H$ is precisely the image of $\MFG^K$.Then \Cref{lem:maximal-euler-suffices} identifies the complement of these closed substacks with the open substack $\MFG^{H/H}$:
  \[
  \MFG^H\setminus\bigcup_{i=1}^r\MFG^{K_i}
  \cong \MFG^{H/H}.
  \]
  The gluing data and the inductive steps can therefore be visualized in \Cref{fig:finite-abelian-gluing}.  Here, $K_1,\dots,K_r$ are maximal subgroups of $H$. The blue node represents the open substack $\mathcal{M}_{fg}^{H/H}$. The inner contours indicate the closed substacks $\mathcal{M}_{fg}^{K_i}$, and the dashed lines indicate induction through smaller subgroups.
  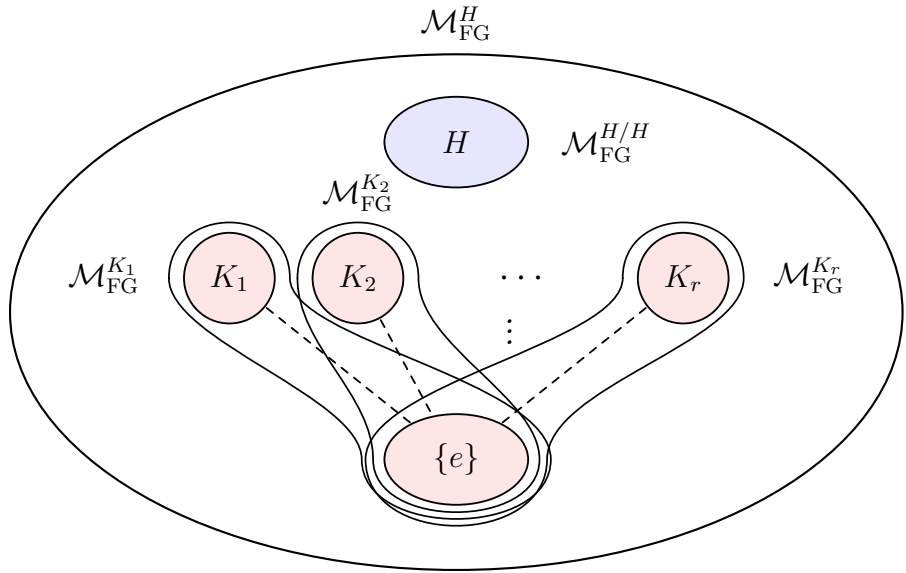
\begin{figure}[H]
    \centering
    \begin{tikzpicture}[
    yscale=0.75, 
        line cap=round,
        line join=round,
        vertex/.style={
            circle,
            draw=black,
            line width=0.7pt,
            minimum size=12mm,
            inner sep=1pt,
            font=\large
        },
        wide vertex/.style={
            vertex,
            ellipse,
            minimum width=19mm,
            minimum height=12mm
        },
        enclosure/.style={
            draw=black,
            fill=none,
            line width=0.65pt
        },
        induction/.style={
            draw=black,
            -, 
            dashed,
            line width=0.7pt
        },
        stack label/.style={font=\large}
    ]

        \draw[enclosure, line width=0.8pt]
            (0,0) ellipse[x radius=5.9cm, y radius=4.55cm];
        \node[stack label] at (0,5.10)
            {$\MFG^{H}$};

        \draw[enclosure]
            (-3.80,0.60)
            .. controls (-3.80,1.91) and (-2.20,1.91) .. (-2.20,0.60)
            .. controls (-2.20,-0.50) and (1.25,-1.20) .. (1.25,-2.60)
            .. controls (1.25,-4.00) and (-1.25,-4.00) .. (-1.25,-2.60)
            .. controls (-1.25,-1.50) and (-3.80,-0.50) .. cycle;

        \draw[enclosure]
            (-2.10,0.60)
            .. controls (-2.10,1.91) and (-0.50,1.91) .. (-0.50,0.60)
            .. controls (-0.50,-0.60) and (1.10,-1.50) .. (1.10,-2.60)
            .. controls (1.10,-3.84) and (-1.10,-3.84) .. (-1.10,-2.60)
            .. controls (-1.10,-1.45) and (-2.10,-0.60) .. cycle;

        \draw[enclosure]
            (2.20,0.60)
            .. controls (2.20,1.91) and (3.80,1.91) .. (3.80,0.60)
            .. controls (3.80,-0.55) and (1.20,-1.32) .. (1.20,-2.60)
            .. controls (1.20,-4.16) and (-1.20,-4.16) .. (-1.20,-2.60)
            .. controls (-1.20,-1.25) and (2.20,-0.55) .. cycle;

        \draw[induction] (-3,0.6) -- (0,-2.6);
        \draw[induction] (-1.3,0.6) -- (0,-2.6);
        \draw[induction] (3,0.6) -- (0,-2.6);

        \node[wide vertex, fill=blue!10] (H) at (0,3.0) {$H$};
        \node[vertex, fill=red!10] (K1) at (-3,0.6) {$K_1$};
        \node[vertex, fill=red!10] (K2) at (-1.3,0.6) {$K_2$};
        \node[vertex, fill=red!10] (Kr) at (3,0.6) {$K_r$};
        \node[wide vertex, fill=red!10] (E) at (0,-2.6) {$\{e\}$};

        \node[font=\Large] at (0.9,0.6) {$\cdots$};
        \node[font=\Large] at (0.70,-0.20) {$\vdots$};

        \node[stack label, anchor=west] at (1.25,3.0)
            {$\MFG^{H/H}$};
        \node[stack label, anchor=east] at (-4.05,0.65)
            {$\MFG^{K_1}$};
        \node[stack label] at (-1.3,2.10)
            {$\MFG^{K_2}$};
        \node[stack label, anchor=west] at (4.05,0.65)
            {$\MFG^{K_r}$};
    \end{tikzpicture}
    \caption{Gluing data for a finite abelian group $H$}
    \label{fig:finite-abelian-gluing}
\end{figure}
\end{remark}

\subsection{The equivariant Landweber exact functor theorem}
Let $N$ be a graded $L_{A}$-module. We define the following three conditions for $N$:
\begin{itemize}[leftmargin=4em]
  \item[$(\mathrm{EU}_0)$] Euler exactness: For every $p$ and every $K\lessdot_p H\leq A$, $(e_{H/K},d_{H/K})$ is an exact pair of zero divisors on $N_{H,(p)}$, i.e.,
    \[
      \ann_{N_{H,(p)}}(e_{H/K})=d_{H/K}N_{H,(p)}, \qquad \ann_{N_{H,(p)}}(d_{H/K})=e_{H/K}N_{H,(p)}.
    \]
  \item[$(\mathrm{EU}_{\infty})$] Euler exactness at all heights: For every $p$, and for every $K\lessdot_p H\leq A$ and every $n\geq0$,the pair
    \[
      \bigl(e_{H/K}^{p^n},d_{H/K}^{(p,n)}\bigr)
    \]
  is exact on $N_H/I_{p,n}N_H$.
  \item[$\mathrm{(LE)}$] Classical Landweber exactness: For every $H\leq A$ and every $n\geq0$, multiplication by $v_{p,n}$ is injective on $\Phi^H N_H/I_{p,n}\Phi^H N$.
\end{itemize}

\begin{lemma}\label{lem:automatic-edge}
Let $K\mathrel{\lessdot_q}H$ with $q\neq p$.  After $p$-localization, the
pair $(e_{H/K},d_{H/K})$ is an exact pair on every $L_{H,(p)}$-module.
\end{lemma}
\begin{proof}
Put $e=e_{H/K}$ and $d=d_{H/K}$. Under $L_H/eL_H\cong L_K$, the class $d$ restricts to
\[
 \res_1^{C_q}(\bar\psi_q)=q
\]
Hence the image of $d$ modulo $e$ is a unit in $L_{K,(p)}$.  Therefore there are $a,b\in L_{H,(p)}$ with
\[
  ae+bd=1.
\]
For any module $M$ and $m\in\ker(e)$, this identity gives $m=bdm\in dM$; the reverse inclusion follows from $ed=0$.  The same argument with $e$ and $d$ interchanged gives $\ker(d)=eM$. 
\end{proof}

\begin{theorem}\label{thm:finite equivariant-LEFT}
Let $A$ be a finite abelian group, and let $N$ be a graded $L_{A}$-module.  The following are equivalent.
\begin{enumerate}[label=\textup{(\roman*)}]
\item\label{item:p-homology} The functors $N\tensor_{L_{A}} MU_*^A(-)$ form a homology theory.
\item\label{item:p-flat} $N\tensor_{L_{A},\eta_L}\Gamma_{A}$ is flat over $L_{A}$ through $\eta_R$.
\item\label{item:p-comodule} $-\tensor_{L_{A}}N$ is exact on graded $(L_{A},\Gamma_{A})$-comodules.
\item\label{item:p-criterion} $N$ satisfies $\mathrm{EU}_{0}+\mathrm{LE}$. Explicitly, for every $K<H\leq A$ with $H/K\cong \bbZ/p$,
\[
\ker(e_{H/K}:N_{H,(p)}\to N_{H,(p)})=d_{H/K}N_{H,(p)}, \qquad \ker(d_{H/K}:N_{H,(p)}\to N_{H,(p)})=e_{H/K}N_{H,(p)},
\]
and, for every prime $p$, every $H\leq A$, and every $n\geq0$,
\[
  v_{p,n}:\Phi^H N_H/I_{p,n}\Phi^H N \longrightarrow \Phi^H N_H/I_{p,n}\Phi^H N
\]
is injective.
\item\label{item:p-allheight} $N$ satisfies $\mathrm{EU}_{\infty}+\mathrm{LE}$.
\end{enumerate}
If $N=R$ is an evenly graded $L_{A}$-algebra, these conditions are also equivalent to flatness of $\Spec R\to\MFG^A$.
\end{theorem}
\begin{proof}
  We prove the theorem via \ref{item:p-flat} $\Longrightarrow$ \ref{item:p-allheight} $\Longrightarrow$ \ref{item:p-criterion} $\Longrightarrow$ \ref{item:p-flat}. The equivalence of \ref{item:p-flat} and \ref{item:p-comodule} is given by Hovey-Strickland\cite[Lemma 2.2]{hovey2005comodules}, and the equivalence of \ref{item:p-homology} and \ref{item:p-comodule} is proved by \Cref{prop:homology-implies-flat} since $(L_A, \Gamma_A)$ is flat.

  \ref{item:p-flat} $\Longrightarrow$ \ref{item:p-allheight}: After the base change via $\MFG^H \to \MFG^A$, for any $H<A$, tensoring with $N_H$ is exact on $(L_{H},\Gamma_{H})$-comodules. Then tensoring with $\Phi^H N$ is exact on $(\Phi^H L, \Phi^H \Gamma)$-comodules. Since there is an equivalence of Hopf algebroids 
  \[
  (L, \Gamma) \simeq (\Phi^H L, \Phi^H \Gamma)
  \]
  since there is an equivalence of stacks associated to these hopf algebroids $\MFG \simeq \MFG^{H/H}$. Then tensoring with $\Phi^H N$ is exact on $(L, \Gamma)$-comodules and the classical Landweber exact functor theorem implies that $N$ satisfies $(\mathrm{LE})$. For any $K\lessdot_p H$ and $n\geq 0$, since $(d_{H/K}^{(p,n)})$ and $(e_{H/K}^{p^n})$ are invariant, there are short exact sequences of $(L_{H,(p)},\Gamma_{H,(p)})$-comodules
  \[
  0\to (e_{H/K}^{p^n}) \to L_{H,(p)}/I_{p,n} \to L_{H,(p)}/(I_{p,n},e_{H/K}^{p^n}) \to 0,
  \] 
  \[
  0 \to (d_{H/K}^{(p,n)}) \to L_{H,(p)}/I_{p,n}\to L_{H,(p)}/(I_{p,n},d_{H/K}^{(p,n)})\to 0.
  \]
  Tensoring the two short exact sequences with $N_H$ together with the isomorphism $(d_{H/K}^{(p,n)}) \cong L_{H,(p)}/(I_{p,n},e_{H/K}^{p^n})$ imply that $(e_{H/K}^{p^n},d_{H/K}^{(p,n)})$ is an exact pair of zero divisors over $N_H$, i.e., $N$ satisfies $(\mathrm{EU}_\infty)$.

  \ref{item:p-allheight} $\Longrightarrow$ \ref{item:p-criterion} follows directly from the definitions, since the condition $(EU_0)$ is part of $(EU_\infty)$.

  \ref{item:p-criterion} $\Longrightarrow$ \ref{item:p-flat}: We prove this by induction on $|H|$. Assuming that $N$ satisfies $\mathrm{EU}_{0}+\mathrm{LE}$. For each $H\leq A$, put
  \[
  M_H=N_H\tensor_{L_{H},\eta_L}\Gamma_{H},
  \]
  with its $\eta_R$-module structure.  For $H=1$, $\mathrm{LE}$ is the classical Landweber condition.which gives the flatness of $M_1$.

  Let $H\neq 1$ and assume the claim for all proper subgroups of $H$.  List all maximal subgroups as $K_1,\ldots,K_m$, and write
  \[
  q_i=|H/K_i|, \qquad e_i=e_{H/K_i}, \qquad d_i=d_{H/K_i}.
  \]
  Then for any prime $p$, $(e_i,d_i)$ is an exact pair on $N_{H,(p)}$ is by hypothesis $(\mathrm{EU}_{0})$ and \cref{lem:automatic-edge}.  Since $\Gamma_{H,(p)}$ is flat over $L_{H,(p)}$, $(e_i,d_i)$ is an exact pair on $M_{H,(p)}$. Note that for every $i$, $(e_i)L_H \cong \ker(L_H\to L_{K_i})$, then \Cref{lem:closed-basechange} gives 
  \[ 
  M_{H,(p)}/e_iM_{H,(p)}\cong N_{K_i,(p)}\tensor_{L_{K_i,(p)},\eta_L}\Gamma_{K_i,(p)} =M_{K_i,(p)}.
  \]
  The induction hypothesis makes this flat over $L_{K_i}\cong L_{H}/(e_i)$.

  By \cref{lem:maximal-euler-suffices}, localization at $e_1\cdots e_m$ is equivalent to the localization at all non-trivial Euler classes. The condition $(\mathrm{LE})$ and the classical Landweber exact functor theorem show that 
  \[
  M_{H,(p)}[\Delta_H^{-1}] \cong (\Phi^H N_{(p)} \otimes_{L_{(p)}} \Gamma_{(p)}) \otimes_{L_{(p)}} \Phi^H L_{(p)}\cong \Phi^H M_{(p)}
  \]
  is flat over $\Phi^H L_{(p)}$. Moreover \cref{lem:tor-exact-pair} gives
  \[
  \Tor_j^{L_{H,(p)}}(L_{K_i,(p)},M_{H,(p)})=0 \qquad(j>0)
  \]
  for every $i$.  Then \cref{lem:finite-gluing} implies that $M_{H,(p)}$ is flat over $L_{H,(p)}$ for all prime $p$, thus $M_H$ is flat over $L_H$.  Taking $H=A$ proves \ref{item:p-flat} and completes the proof of the theorem.
\end{proof}
\begin{corollary}\label{cor:representation}
  Every graded $L_A$-module $N$ satisfying $(\mathrm{EU}_0)+(\mathrm{LE})$ admits a representing genuine $A$-spectrum $E_A$ with coefficients
  \[
  \pi_*^H E_A \cong L_H\tensor_{L_A}N.
  \]
\end{corollary}
\begin{proof}
  When $N$ satisfies $(\mathrm{EU}_0)+(\mathrm{LE})$, we can define an $RO(A)$-graded homology theory $E$ by 
  \[
  E_V(X):= (MU_A)_0(S^{-V}\wedge X)\otimes_{L_A} N
  \]
  for all genuine $A$-spectra $X$ and $V\in RO(A)$. The Brown representability theorem \cite[Chapter XIII, Corollary 3.5]{may1996equivariant} gives a genuine $A$-spectrum $E_A$ such that $E_V(X)\cong \pi_V^A(E_A \wedge X)$. 
\end{proof}
\begin{corollary}\label{cor:spectra level representation}
  Let $E$ be a complex oriented $A$-spectrum $E$. Then 
  \[
  E_*(X) \cong (MU_A)_*(X)\otimes_{L_A} \pi_*^A E
  \]
  if and only if $\pi_*^A E$ is landweber exact, and for any $B\leq A$, $\pi_*^B E \cong L_B\otimes_{L_A} \pi_*^A E$.
\end{corollary}
\begin{proof}
  The corollary is followed by \Cref{thm:finite equivariant-LEFT,cor:representation}.
\end{proof}

\section{Landweber exactness for abelian compact Lie groups}\label{sec:compact}
In this section, we prove an equivariant Landweber exact functor theorem for abelian compact Lie groups. Let $G$ be a rank $m$ abelian compact Lie group. Fix a decomposition
\[
G=G^\circ \times \pi_0G = T\times A,\qquad T=\mathbb T^m,
\]
where $A$ is finite abelian. Note that every Euler class $e_\tau$ of $\tau\in (T)^*\cong \bbZ^m$ is a non-zero divisor. Then for a $G$-module $N$, it suffices to consider the following additional condition:
\begin{itemize}[leftmargin=4em]
  \item[$(\mathrm{RE})$] Regular Euler classes: For any $r\leq m$ and any linearly independent $\{\tau_i\}_r \subset T^* \subset G^*$, the sequence $(e_{\tau_1},\cdots, e_{\tau_r})$ is a regular sequence over $N$.
\end{itemize}
Besides, $(\mathrm{EU}_0)$ and $(\mathrm{LE})$ should be replaced by the following condition:
\begin{itemize}[leftmargin=4em]
  \item[$(\mathrm{EU})$] For every $p$ and every $K\lessdot_p H\leq A$, $(\pi^* e_{H/K},\pi^* d_{H/K})$ is an exact pair of zero divisors on $N_{T\oplus H,(p)}$ via the pullback along $\pi^*\colon T\oplus H \to H$.
  \item[$\mathrm{(LE)}$] For every prime $p$, every closed subgroup $B\leq G$ and every $n\geq0$, multiplication by $v_{p,n}$ is injective on $\Phi^B N_B/I_{p,n}\Phi^B N_B$.
\end{itemize}
Note that there are infinitely many closed substacks of $\MFG^G$ and we can't use the same argument as in \Cref{lem:finite-gluing} to gluing them together one by one. In this case, we need the following gluing lemma.

\begin{lemma}[Arbitrary Euler gluing]\label{lem:compact-infinite-gluing}
Let $D$ be a commutative ring, $M$ a $D$-module, and $S$ the multiplicative set generated by an arbitrary family $\{f_\alpha\}_{\alpha\in J}$. Suppose that $S^{-1}M$ is flat over
$S^{-1}D$ and, for every $\alpha\in J$,
\[
M/f_\alpha M\text{ is flat over }D/(f_\alpha),
\qquad
\operatorname{Tor}_i^D(D/(f_\alpha),M)=0\quad(i=1,2).
\]
Then $M$ is flat over $D$.
\end{lemma}
\begin{proof}
For any $\alpha \in J$, let $N_\alpha$ be an $D/f_\alpha$-module. Choose a short exact sequence
\[
0\to K \to \bigoplus D/f_\alpha \to N \to 0,
\]
then \cite[Lemma 10.99.8]{stacks-project} implies that $\operatorname{Tor}_1^D(N,M)=0$, and 
\[ 
\operatorname{Tor}_2^D(N,M)\cong \operatorname{Tor}_1^D(K,M)\cong 0.
\]
Now for any $D$-module $N$, consider the short exact sequence of $D$-modules
\[
0\longrightarrow N[f_\alpha]\longrightarrow N\xrightarrow{f_\alpha}f_\alpha N
\longrightarrow0,
\qquad
0\longrightarrow f_\alpha N\longrightarrow N\longrightarrow N/f_\alpha N
\longrightarrow0.
\]
The long exact sequences of $\operatorname{Tor}$ associated to these two short exact sequences give a monomorphism followed by an isomorphism
\[
\operatorname{Tor}_1^D(N,M)\lhook\joinrel\longrightarrow
\operatorname{Tor}_1^D(f_\alpha N,M)\xrightarrow{\ \cong\ }\operatorname{Tor}_1^D(N,M).
\]
The composition is induced by the multiplication by $f_\alpha$ on $N$. Consequently every $f_\alpha$, and hence every $s\in S$, acts injectively on $\operatorname{Tor}_1^D(N,M)$. Note that the flat base change for $\operatorname{Tor}$ and the flatness of $S^{-1}M$ give
\[
S^{-1}\operatorname{Tor}_1^D(N,M)\cong
\operatorname{Tor}_1^{S^{-1}D}(S^{-1}N,S^{-1}M)=0.
\]
Thus $\operatorname{Tor}_1^D(N,M)=0$ for every $D$-module $N$, which proves the flatness of $M$ over $D$.
\end{proof}

Now we can define the gluing data for abelian compact Lie groups. For any closed subgroup $K< H \leq G$ such that $H/K$ is cyclic of prime order, define
\[
e_{H/K}=e_{\chi_{H/K}},\qquad
d_{H/K}=\operatorname{tr}_K^H(1),
\]
where $\chi_{H/K}:H\to\mathbb T$ is the resulting quotient character. In particular, for any given $H$, the set 
\[
\{K<H:H/K\text{ is cyclic of prime order}\}
\]
is finite, which we also denote by $\Max(H)$. Then with the same argument, \Cref{lem:adapted-group,prop:universal-edge,prop:d-invariant,lem:closed-basechange,lem:automatic-edge} hold for abelian compact Lie groups $H$ and closed subgroups $K\leq H$. It remains to show that for every subgroup $H\leq G$ and every $K\in\Max(H)$, the pair the pair $(e_{H/K},d_{H/K})$ is exact on $N_{H,(p)}$ for an $L_G$-module $N$ satisfying $(\mathrm{RE})$ and $(\mathrm{EU})$. 

\begin{lemma}\label{lem:compact-exact-regular}
Suppose that $(e,d)$ is an exact pair on a $D$-module $M$.
\begin{enumerate}[label=\textup{(\roman*)},leftmargin=3em]
\item If $f$ is injective on $M$, then $(e,d)$ is exact on $M/fM$.
\item Every $M$-regular sequence is regular on $M/eM$.
\end{enumerate}
\end{lemma}
\begin{proof}
Let $\mathcal P(M)$ be the chain complex with every term equal to $M$ and differentials alternating between $e$ and $d$. It is exact by hypothesis. If $f$ is injective on $M$, then
\[
0\longrightarrow\mathcal P(M)\xrightarrow{f}\mathcal P(M) \longrightarrow\mathcal P(M/fM)\longrightarrow0
\]
is a short exact sequence of complexes. The long exact sequence of homology groups proves \textup{(i)}.

For \textup{(ii)}, the map
\[
M/eM\longrightarrow M,\qquad \overline m\longmapsto dm,
\]
is injective. Thus the first element of an $M$-regular sequence is injective on $M/eM$. After quotienting $M$ by that element, part \textup{(i)} preserves the exact pair. Iterating this argument proves successive injectivity for the whole sequence on $M/eM$.
\end{proof}

\begin{lemma}\label{lem:compact exactness condition}
Let $N$ be an $L_G$-module satisfying $(\mathrm{RE})$ and $(\mathrm{EU})$. 
\begin{enumerate}[label=\textup{(\roman*)},leftmargin=3em]
  \item For every prime $p$, every closed subgroup $H\leq G$, and every $K\in\Max(H)$, the pair $(e_{H/K},d_{H/K})$ is exact on $N_{H,(p)}$. 
  \item For every closed subgroup $H\leq G$, whenever $\lambda_1,\lambda_2, \dots, \lambda_r\in H^*$ are linearly independent in $H^*\otimes \mathbb Q$, there Euler classes form a regular sequence on $N_H$.
\end{enumerate}
\end{lemma}
\begin{proof}
  (i) For the first statement, it suffices to consider the case $H/K\cong C_p$. let $A_H$ and $A_K$ be the projection of $H$ and $K$ to $A$, respectively. 
  
  (a) If $A_H=A_K=0$, then $K< H\leq T\cong \bbT^m$, and we can choose a basis such that the inclusion $K\to H$ is given by an inclusion $\bbT^r\times K'\to \bbT^r\times H'\subset \bbT^r\times \bbT^{m-r}$, where $H'$ is finite and elements in $(\bbT^r)^*$ is linearly independent to elements in $(H')^*$. Then we can quotient those Euler classes for $\tau\in (\bbT^r)^*$ and assume that $H$ is finite.  Let 
  \[
  \Lambda_H:=\ker (T^* \to H^*), \Lambda_K:=\ker ( T^*\to K^*).
  \]
  We can choose a generator of $\Lambda_K$ given by $\{\lambda_1,\dots, \lambda_r\}$ such that the $r$-tuple $(\lambda_1,\dots,\lambda_r)$ are linearly independent, and $\Lambda_H=\langle p\lambda_1,\lambda_2,\dots, \lambda_r\rangle$. Then
  \[
  N_H \cong N_T/(e_{p\lambda_1},e_{\lambda_2}, \dots, e_{\lambda_r}),\quad N_K \cong N_T/(e_{\lambda_1},e_{\lambda_2}, \dots, e_{\lambda_r})
  \]
  and $e_{H/K}=\operatorname{Res}_H^T e_{\lambda_1}$. Since $e_{p\lambda_1}$ restrict to $0$ in $L_H$, $e_{p\lambda_1}=e_{\lambda_1}\cdot \psi_p$ such that $d_{H/K}=\operatorname{Res}_H^T \psi_p$. Since $e_{\lambda_1}$ is a nonzero divisor in $N_T$. Thus $(e_{\lambda_1}, \psi_p)$ is an exact pair of zero divisor in $N_T/e_{p\lambda_1}$. Note that $(\mathrm{RE})$ implies that $(e_{\lambda_2},\dots, e_{\lambda_r})$ is regular on $N_T/(e_{p\lambda_1})$, part~\textup{(i)} of \Cref{lem:compact-exact-regular} implies that $(e_H/K,d_{H/K})$ is exact on $N_H$. 

  (b) If $A_H=A_K\neq 0$, then $H\subset L\times B$ and $K\subset  W\times B$ for some $W\lessdotp L< T$ and $B\leq A$. For any subgroup $B\leq A$, choose a chain
  \[
  B=B_0<B_1<\cdots<B_r=A
  \]
  in which each $B_i$ is a maximal subgroup of $B_{i+1}$, so that $B_{i+1}/B_i\cong C_{q_i}$ for some prime $q_i$. At each step, restriction from $T\times B_{i+1}$ to $T\times B_i$ kills the Euler class $e_{B_{i+1}/B_i}\in L_{T\times B_{i+1}}$, which belongs to an exact pair. Applying part~\textup{(ii)} of \Cref{lem:compact-exact-regular} successively, we may therefore assume without loss of generality that $B=A$. Let 
  $$\Lambda_{K}:=\ker (G^* \to {K}^*)=\langle \lambda_1,\lambda_2,\dots,\lambda_r\rangle, \quad \Lambda_{H}=\langle p\lambda_1,\lambda_2,\dots,\lambda_r\rangle.$$ 
  Here $\lambda_i=\tau_i\alpha_i$ for some $\tau_i\in T^*$ and $\alpha_i\in A^*$. 
  Since $A$ is finite, we can choose $s\in \mathbb Z$ such that $s A^*=0$. Then $$e_{s\lambda_i}=e_{s\tau_i}.$$
  Condition $(\mathrm{RE})$ ensures that the sequence $(e_{s\tau_1},\dots,e_{s\tau_r})$ is regular on $N$. For each $i$, we have $e_{s\lambda_i}=e_{\lambda_i}u_i$ for some $u_i\in L_G$, and hence the sequence $(e_{\lambda_1},\dots,e_{\lambda_r})$ is also regular on $N$. As in the case $K<H\leq T\cong\bbT^m$, it then follows from part~\textup{(i)} of \Cref{lem:compact-exact-regular} that $(e_{H/K},d_{H/K})$ is an exact pair on $N_H$.

  (c) Now assume that $A_H/A_K\cong C_p$. If $H\cong \bbT^r\times A$ and $K\cong \bbT^r\times B$ for some $B\lessdotp A$, then $(e_{H/K},d_{H/K})$ is exact on $N_H$, which is followed by $(\mathrm{EU})$ and \Cref{lem:compact-exact-regular}. In general, let $C_H=T\times A_H$ and $C_K=T\times A_K$. Then the pullback of $e_{A_H/A_K}$ belongs to an exact pair in $N_{C_H}$. As the case (b) above, we can choose a regular sequence of Eulers $(e_1,\dots,e_r)$ such that 
  \[
  N_H\cong N_{C_H}/(e_1,\dots, e_r).
  \]
  Since $K=H\cap C_K$, the quotient character $C_H \to A_H/A_K\cong C_p$ restricts to the quotient $H\to H/K$, thus $e_{H/K}$ is the restriction of $e_{C_H/C_K}$ and thus $(e_{H/K},d_{H/K})$ is exact on $N_H$ by part (i) of \Cref{lem:compact-exact-regular}.

  (ii) Let $A_H$ be the projection of $H$ to $A$ and let $C=T\times A_H$. Let 
  \[
  \Lambda_H =\ker (C^*\to H^*).
  \]
  $\Lambda_H$ is free since an element in $\Lambda_H$ belongs to \(A_H^*\) and vanishes on $H$ must be trivial since $H\to A_H$ is surjective. Then we can choose a basis $\chi_1,\dots, \chi_s$ such that $e_i=e_{\chi_i}$ form a regular sequence on $N_C$, and 
  \[
  N_H \cong N_C/(e_1,\dots, e_s).
  \]
  Suppose
  \[
  \lambda_1,\ldots,\lambda_r\in H^*
  \]
  are linearly independent in \(H^*\otimes\mathbb Q\), and choose lifts
  \[
  \widetilde\lambda_1,\ldots,\widetilde\lambda_r\in C^*.
  \]
  Then
  \[
  \widetilde\lambda_1,\ldots,\widetilde\lambda_r, \chi_1,\ldots,\chi_s
  \]
  are linearly independent in $ C^*\otimes\mathbb Q\cong T^*\otimes\mathbb Q$, and the sequence of Euler classes associated to these characters is regular on $N_C$. Then $e_{\lambda_1}, \dots, e_{\lambda_r}$ is regular on $N_H$. 
\end{proof}
As a result, for any subgroup $H$ and $\lambda\in H$ has infinite order or prime order, 
\[
  \operatorname{Tor}_j^{L_{H,(p)}}
  \bigl(L_{H,(p)}/(e_\tau),N_{H,(p)}\bigr)=0
  \qquad\text{for all }j>0.
\]
Now we can prove the following equivariant Landweber exact functor theorem for abelian compact Lie groups.
\begin{theorem}[Equivariant LEFT for abelian compact Lie groups]\label{thm:compact-LEFT}
Let $G\cong T\times A$, where $T=\bbT^m$ and $A$ is finite. For a graded $L_G$-module $N$, the following are equivalent:
\begin{enumerate}[label=\textup{(\roman*)},leftmargin=3em]
\item The functors $ N\tensor_{L_{G}} MU_*^G(-) $ form a homology theory.
\item The functor
\[
N\otimes_{L_G}- \colon \operatorname{Comod}_{(L_G,\Gamma_G)} \longrightarrow\operatorname{Mod}_R
\]
is exact.
\item $M_G(N):= N\otimes_{L_G,\eta_L}\Gamma_G$ is flat over $L_G$ through $\eta_R$.
\item $N$ satisfies the conditions $$(\mathrm{RE})+(\mathrm{EU})+ (\mathrm{LE}).$$
\end{enumerate}
\end{theorem}
\begin{proof}
  The equivalence of \textup{(i)}, \textup{(ii)}, and \textup{(iii)} follows by the same argument as in the case of finite abelian groups.

  \textup{(iii)} $\Longrightarrow$ \textup{(iv)}: Assuming that $M_G(N)$ is flat over $L_G$, it suffices to verify condition $(\mathrm{RE})$, as the remaining two conditions follow by the same arguments as in the finite case. If a sequence $(e_1,\dots, e_r)$ of Euler classes is regular on $L_G$, then 
  \[
  e_{i+1}: L_G/(e_1,\dots, e_i)\to L_G/(e_1,\dots, e_i)
  \]
  is injective. Tensoring with $M_G(N)$ shows that $(e_1,\dots, e_r)$ is regular on $M_G(N)$ since $(e_i)$ is an invariant ideal. Then $(e_1,\dots, e_r)$ is regular on $N$ since $\Gamma_G$ is faithful flat over $L_G$.

  \textup{(iv)} $\Longrightarrow$ \textup{(iii)}: By \Cref{lem:compact exactness condition}, $(\mathrm{RE})+(\mathrm{EU})$ implies that for any closed subgroup $H\leq G$ and $K\in \Max (H)$, $(e_{H/K},d_{H/K})$ is exact on $N_H$. We now prove that $M_G(N)$ is flat over $L_G$ by induction on the pairs
  \[
  (\dim H,|\pi_0 H|),
  \]
  ordered lexicographically, where $H$ ranges over the closed subgroups of $G$. Every proper closed subgroup $K<H$ has a strictly smaller pair: if $\dim K=\dim H$, then $K^\circ=H^\circ$, and hence $|\pi_0 K|<|\pi_0 H|$. Let 
  \[
  \mathcal E_H := \{e_\tau: \tau\in H^* \text{ has infinite order or prime order.}\}
  \]
  Let $S_H$ be their multiplicative closure.  Then
  \[
  S_H^{-1}L_H=\Phi^HL.
  \]
  Indeed, only finite-order characters of composite order remain to be checked. If $\chi$ has order $n>1$, choose a prime $q\mid n$ and set
  \[
  \lambda=(n/q)\chi.
  \]
  Then $\lambda$ has order $q$, so $e_\lambda$ is inverted in $S_H^{-1}L_H$. Since $\lambda$ vanishes on $\ker(\chi)$,  $e_\lambda\in(e_\chi)$. Thus invertibility of $e_\lambda$ forces invertibility of $e_\chi$.

  The inductive step follows by applying \Cref{lem:compact-infinite-gluing} to $\mathcal E_H$, and the same arguments as in the finite case complete the proof.
\end{proof}

\begin{corollary}\label{cor:compact representation}
  Let $E$ be a complex oriented $G$-spectrum $E$. Then 
  \[
  E_*(X) \cong (MU_G)_*(X)\otimes_{L_G} \pi_*^G E
  \]
  if and only if $\pi_*^G E$ satisfies $(\mathrm{RE})+(\mathrm{EU})+ (\mathrm{LE})$, and for any $H\leq G$, $\pi_*^H E \cong L_H\otimes_{L_G} \pi_*^G E$.
\end{corollary}

\begin{corollary}[from nonequivariant to equivariant]\label{cor:ordinary-implies-equivariant}
  Let $M$ be a nonequivariant Landweber exact $L$-module. Then $L_G\otimes_L M$ is $G$-equivariant Landweber exact. Consequently
  \[
  X\longmapsto MU_*^G(X)\tensor_LM
  \]
  is a homology theory on genuine $A$-spectra.
\end{corollary}
\begin{proof}
  It suffices to verify the conditions $(\mathrm{RE})+(\mathrm{EU})+ (\mathrm{LE})$.

  The condition ($\mathrm{RE}$): Let $\tau_1,\dots,\tau_r\in T^*$ be linearly independent. Then $(e_{\tau_1},\dots,e_{\tau_r})$ is a regular sequence on $L_G$. For $0\leq i\leq r$, put
  \[
  Q_i=L_G/(e_{\tau_1},\dots,e_{\tau_i}), \qquad Q_0=L_G.
  \]
  For each $0\leq i<r$, the short exact sequence
  \[
  0\longrightarrow Q_i \xrightarrow{\,e_{\tau_{i+1}}\,} Q_i \longrightarrow Q_{i+1}\longrightarrow 0
  \]
  splits over $L$, as its cokernel $Q_{i+1}$ is free over $L$. Consequently, it remains exact after tensoring over $L$ with $M$. Under the natural identification
  \[
  Q_i\otimes_L M \cong (L_G\otimes_L M)/(e_{\tau_1},\dots,e_{\tau_i}),
  \]
  the resulting map is multiplication by $e_{\tau_{i+1}}$. Thus $(e_{\tau_1},\dots,e_{\tau_r})$ is a regular sequence on $L_G\otimes_L M$.

  The condition ($\mathrm{EU}$): For any $K\lessdotp H\leq\pi_0G$, the pair $(e,d)=(e_{H/K},d_{H/K})$ is exact on $L_{G^\circ\times H}$; equivalently, the complex
  \[
  \cdots \xrightarrow{\,e\,} L_{G^\circ\times H} \xrightarrow{\,d\,} L_{G^\circ\times H} \xrightarrow{\,e\,} L_{G^\circ\times H} \xrightarrow{\,d\,} \cdots
  \]
  is exact. This complex is split exact over $L$ and therefore remains exact after tensoring over $L$ with $M$. It follows that $(e,d)$ is an exact pair on $L_{G^\circ\times H}\otimes_L M$.

  The condition $(\mathrm{LE})$ follows from the Landweber exactness of $M$ and the flatness of $\Phi^H L$ over $L$ for every closed subgroup $H$.
\end{proof}

As a result, we have the following proposition, which provides a construction of equivariant Landweber exact spectra from non-equivariant ones. This resolves \cite[Conjecture 2.18]{wisdom2024properties}.
\begin{proposition}\label{prop:conjecture-in-Wisdom}
  If $E$ is a Landweber exact spectrum, then $MU_A \wedge_{MU} E$ is an $G$-Landweber exact spectrum.
\end{proposition}
\begin{proof}
  According to \cite[Theorem 1.1]{may1998equivariant}, for all closed subgroup $H\leq G$, 
  $$\pi_*^H MU_G\wedge_{MU} E \cong \pi_*^H MU_H\wedge_{MU} E \cong L_H\otimes_L \pi_*E \cong L_H\otimes_{L_G} (L_G\otimes_L \pi_*E).$$
  Then the proposition follows by \Cref{cor:ordinary-implies-equivariant}.
\end{proof}

\section{Equivariant elliptic cohomology}\label{sec:equivariant elliptic}

In this section, let $G$ be an abelian compact Lie group. We construct a $G$-equivariant formal groups $\widehat{E_G}$ over $\hom(G^*,E)$ from any elliptic curve $E/S$. If the representation morphism $S\to \mathcal M_{ell}$ of $E$ is flat, we prove that the $G$-forml group $\widehat{E_G}$ is also flat over $\MFG^G$, see \Cref{prop:elliptic formal group is exact}. Moreover, for every elliptic curve $E/S$ which is \'{e}tale over $\mathcal M_{ell}$, we construct a $G$-formal group associated to the equivariant elliptic cohomology theory defined by Gepner-Meier \cite{gepner2023equivariant}, which is isomorphic to $\widehat{E_G}$, see \Cref{prop:compactible of elliptic formal groups}.

\subsection{Equivariant formal groups from elliptic curves}
Let $S$ be a scheme, and let $p:E\to S$ be an elliptic curve over $S$ such that the representation map $S\to \mathcal{M}_{ell}$ is flat. Here $\mathcal{M}_{ell}$ is the classical moduli stack of elliptic curves.
Consider the $S$-group schemes
\[
S_G:=\hom_S(G^*,E) \to S, \quad f\mapsto p(f(\epsilon)).
\]
Let
\[
E_G:=E\times_S S_G
\]
be the elliptic curve over $S_G$ given by the pullback of $E$, then there is a tautological map of group schemes over $S_G$
\[
\Phi_G:G^*\times S_G\longrightarrow E_G, \quad (\chi, f) \mapsto (f,f(\chi)).
\]
For any finite dimensional complex $G$-representation $V$, choose a decomposition $V=\bigoplus\limits_{\chi\in F_V}\chi$ into irreducible representations, then we can define the divisor on $E_G$ associated to $V$ as 
\[
D_V=\sum_{\chi\in F_V} [\Phi(\chi)] \subset E_G,
\]
and let $p_V:D_V\to S_G$ be the projection. Completion $E_G$ along the directed system given by all finite dimensional $G$-representations $V$, we can difine a formal group $$\widehat{E_G}:=\operatorname{colim}_{V} D_V$$ over $S_G$, together with a canonical map 
\[
\phi_G:G^*\times S_G\longrightarrow \widehat{E_G}.
\]
This is a generalization of Strickland's construction in \cite[Section~7.2]{strickland2011multicurves} and gives a $G$-equivariant formal group associated to an Elliptic curve.

\begin{lemma}\label{lem:G-formal group associated to elliptic curves}
  For any abelian compact Lie group $G$, $\phi_G:G^*\to \widehat{E_G}$ is a $G$-equivariant formal group. 
\end{lemma}
\begin{proof}
  Let $\mathcal{L}_\epsilon=\mathcal{O}_{E_G}(-[\Phi_G(\epsilon)])$ be the ideal sheaf of the zero section of $E_G$. Let 
  \[(\chi_1,\chi_2,\cdots)\in (G^*)^{\mathbb N}\]
  be a complete flag for $G$ such that $\chi_1=\epsilon$, i.e., a sequence of characters such that every every character appears infinitely often. 
  For any $n\geq 1$, let 
  \[
  D_n=D_{\sum_{i=1}^{n}\chi_i}=\sum_{i=1}^n [\Phi_G(\chi)],
  \]
  then $\widehat{E_G}=\operatorname{colim}_n D_n$. 
  
  For any $n$, The frame bundle $Q_n:=\operatorname{Isom}_{D_n}(\mathcal{O}_{D_n},\mathcal{L}_\epsilon|_{D_n})$ associated to the line bundle $\mathcal{L}_\epsilon$ is a $\mathbb{G}_m$-torsor over $D_n$, thus $Q_n\to D_n$ is affine, smooth, and faithfully flat of finite representation. Then the Weil restriction $P_n:=\mathbf{R}_{D_n/S_G}(Q_n)$ of $Q_n$ is affine, smooth, faithfully flat of finite representation over $S_G$, see \cite[Proposition~3.5 (v)(ix), Proposition~3.8 (xvi)]{Rydh01072011}. 
  Here, by definition, $P_n(T)=\operatorname{Isom}_{\mathcal{O}_{(D_n)_T}}(\mathcal{O}_{(D_n)_T}, (\mathcal{L}_\epsilon)_{(D_n)_T})$ for any $S_G$-shceme $T$. Then $id_{P_n}\in P_n(P_n)$ corresponds to an isomorphism 
  \[
  \mathcal{O}_{D_n\times_{S_G} P_n}\cong (\mathcal{L}_\epsilon)_{D_n\times_{S_G} P_n}
  \]
  The restriction of the isomorphism from $D_{n+1}$ to $D_n$ gives rise to a morphism $P_{n+1}\to P_n$. 

  Let $\Spec R \to S_G$ be an affine chart of $S_G$, then $P_n$ pulls back to an affine scheme $\Spec U_n$ over $\Spec R$, which is faithfully flat. Moreover, $\mathcal{O}_{(D_n)_{U_n}}\cong (\mathcal{L}_\epsilon)_{(D_n)_{U_n}}$. Since $(D_n)_{U_n}$ is finite over $\Spec U_n$, it is affine and assume that $(D_n)_{U_n}\cong\Spec H_n$. Then we have an isomorphism of $H_n$-modules
  \[
  \Gamma(\Spec H_n, \mathcal{O}_{(D_n)_{U_n}})\cong \Gamma(\Spec H_n, (\mathcal{L}_\epsilon)_{(D_n)_{U_n}}).
  \]
  Let $U=\operatorname{colim}_n U_n$, then it follows from \cite[Lemma 10.39.20]{stacks-project} that $U$ is faithfully flat over $R$. Since $\epsilon$ occurs infinitely often in the complete flag, we may restrict to the cofinal set of indices $m$ for which $\chi_{m+1}=\epsilon$ Then base change all those sheaves over $U_m$ to $U$ gives a short exact sequence of $H_m\otimes_{U_m} U$-modules
  \[
  0\to \Gamma((D_m)_U, (\mathcal{L}_\epsilon)_{(D_m)_{U}}) \to \Gamma((D_{m+1})_U, \mathcal{O}_{(D_{m+1})_{U}}) \to \Gamma((D_0)_U, \mathcal{O}_{(D_0)_U})\to 0,
  \]
  here $(D_0)_U\cong (\Spec R)_U\cong \Spec U$.
  Taking inverse limits to this exact sequence gives 
  \[
  0\to \lim_{\leftarrow}\Gamma((D_m)_U, (\mathcal{L}_\epsilon)_{(D_n)_{U}}) \to \Gamma((\widehat{E_G})_U, \mathcal{O}) \to U \to 0.
  \]
  Thus the argumentation ideal $I_\epsilon \cong \lim\limits_{\leftarrow}\Gamma((D_m)_U, (\mathcal{L}_\epsilon)_{(D_m)_{U}})$, which is isomorphic to $\Gamma((\widehat{E_G})_U, \mathcal{O})$.
\end{proof}

\begin{remark}
  The base scheme $S_G=\hom_S(G^*,E)$ here is an universal object: Given any $S$-schemes $T$ together with the following diagram.
        $$\begin{tikzcd}
            G^*\times T\arrow[r,"\Theta"]\arrow[rd,"\operatorname{pr}"] & T\times_S E\arrow[r]\arrow[d] & E\arrow[d]\\
            & T\arrow[r,"f"] & S.
        \end{tikzcd}$$ 
        Then for every $\chi\in G^*$, $\Theta_\chi$ is a section $T\to T\times_S E$, which is equivalent to a map of schemes $T\to E$ over $S$. Thus the group homomorphism $\Theta$ is equivalent to a map of abelian groups $G^*\to \hom_S(T, E)$, which is determined by a map of schemes $T\to S_G$. 
\end{remark}

Since $\phi_G:G^*\to \widehat{E_G}$ is a $G$-equivariant formal group, it is represented by $S_G\to \MFG^G$, and we also denote it by $\phi_G$.  
Let $\mathcal{O}_E(-[0])$ be the ideal sheaf of the zero section $S\to E$, and let $s_0\in \Gamma(E, \mathcal{O}_E([0]))$ be the element corresponding to the natural inclusion $\mathcal{O}_E(-[0])\to \mathcal{O}_E$. For any $\chi\in G^*$, we can define the Euler class $e_\chi$ of $\phi_G$ as
\[
e_\chi:=\operatorname{ev}_{\chi}^* s_0 \in \Gamma(S_G, \operatorname{ev}_\chi^*\mathcal{O}_E([0])),
\]
which is the pullback of $s_0$ via the map
\[
\operatorname{ev}_\chi: S_G\overset{\operatorname{id}\times \chi}{\longrightarrow} G^*\times S_G \overset{\phi_G}{\longrightarrow} E_G \longrightarrow E.
\]
Note that a $G$-formal group $G^*\to X$ admits a coordinate if and only if the ideal sheaf associated to the zero section is trivial, that is, 
$\mathcal{O}_X(-[\epsilon]) \cong \mathcal{O}_{X},$
see \cite[Proposition~2.6]{hausmann2023invariant}.
If $\phi_G$ is coordinatizable, pulling back $s_0$ to $E_G$ and then restricting to $\widehat{E_G}$ gives a section $s_\epsilon\in \Gamma(\widehat{E_G}, \mathcal{O}_{\widehat{E_G}}([\epsilon]))$, which gives such a trivialization $\mathcal{O}_{\widehat{E_G}}(-[\epsilon]) \cong \mathcal{O}_{\widehat{E_G}}$. The Euler class we define above is then given by $e_\chi=(\phi_G)_\chi^*(s_\epsilon)$, which is exactly the classical one.

\begin{lemma}\label{lem:change of groups}
  For any closed subgroup $K\leq H \leq G$, there is an exact sequence of group schemes over $S$
  \[
  0\longrightarrow S_K \longrightarrow S_H \longrightarrow S_{H/K} \longrightarrow 0.
  \]
  In particular, $S_H\to S_{H/K}$ is faithfully flat, and there is a pullback square
  \[
  \begin{tikzcd}
    S_K \ar[r]\ar[d,"\phi_K"] & S_H \ar[d,"\phi_H"]\\
    \MFG^K \ar[r] & \MFG^H.
  \end{tikzcd}
  \]
\end{lemma}
\begin{proof}
  For any abelian compact Lie group $G=\bbT^m\times \prod\limits_{i=1}^k\bbZ/n_i$, $$S_G=E^m\times_S E[n_1] \times_S\cdots\times_S E[n_k].$$ So $S_G$ is flat over $S$, see \cite[Theorem 2.3.1]{KatzMazur1985}. Moreover, as a functor, $$S_G(T)=\hom_{\Ab}(G^*, E(T))$$ for evert $S$-scheme $T$. Pontryagin duality implies that $0\to (H/K)^*\to H^*\to K^*\to 0$ is exact, so 
  \[
  0 \longrightarrow S_K \longrightarrow S_H \longrightarrow S_{H/K} 
  \]
  is exact. It suffices to show that $S_H\to S_{H/K}$ is faithfully flat. 
  
  Choosing a representation $(H/K)^*\cong \bbZ^r \bigoplus \oplus_i\bbZ/l_i$, we can construct a chain of subgroups $K= K_1\leq\cdots\leq K_t=H$ such that $(K_{i+1}/K_i)^*\cong \bbZ$ or $\bbZ/l_i$. Therefore, without lose of generality, we can assume that $(H/K)^*\cong \bbZ$ or $\bbZ/l_i$. 
  In either case, $S_H\to S_{H/K}$ is induced by the inclusion $(H/K)^*\hookrightarrow H^*$. If $(H/K)^*\cong\bbZ$, then $S_{H/K}\cong E$. After choosing a suitable basis of $H^*$, $S_H\to S_{H/K}$ is given by projection multiplication by some integer on $E$. If $(H/K)^*\cong\mathbb Z/\ell$, then $S_{H/K}\cong E[\ell]$, and the morphism is induced by the corresponding homomorphisms between torsion subgroup schemes of $E$. In both cases, these morphisms are faithfully flat.

  Note that $S_K=\ker(S_H\to S_{H/K})$, which is exactly the vanishing locus of the Euler classes $e_\chi$ for $\chi\in (H/K)^*$. Since $\MFG^K\to\MFG^H$ is also the closed substack given by the common zero locus of those Euler classes, 
  \[
  \begin{tikzcd}
    S_K \ar[r]\ar[d,"\phi_K"] & S_H \ar[d,"\phi_H"]\\
    \MFG^K \ar[r] & \MFG^H
  \end{tikzcd}
  \]
  is a pullback square.
\end{proof}

Now fix a fpqc chart $\Spec R\to S_G$ such that 
$$\phi_G|_{\Spec R}: \Spec R \to S_G \to \MFG^G$$
admits a coordinate. Then $\phi_G|_{\Spec R}$ factors through $\Spec L_G$ such that $R$ can be a $L_G$-algebra. For any closed subgroup $H\leq G$, let 
\[
\Spec R_H= \Spec R \times_{S_G} S_H,
\]
then $\phi_H|_{R_H}$ factors through $\Spec L_H$, and \Cref{lem:change of groups} implies that 
$$R_H\cong L_H\otimes_{L_G} R.$$

\begin{lemma}\label{lem:condition Re and Eu for elliptic}
  Let $G=T\times A$ for some torus $T$ and finite abelian group $A$.
  \begin{itemize}
    \item[(i)] For any linearly independent $\{\tau_i\}\subset T^*$, the sequence $(e_{\tau_i})$ is a regular sequence over $R$.
    \item[(ii)] For any prime $p$ and closed subgroups $K\leq H \leq G$ such that $H/K\cong \bbZ/p$, $(e_{H/K},d_{H/K})$ is an exact pair on $R_H$.
  \end{itemize}
\end{lemma}
\begin{proof}
  (i)  For any $\tau \in T^*$, \Cref{lem:change of groups} implies that there is an exact sequence
  \[
  0\to S_{\ker\tau}\longrightarrow S_G \overset{\operatorname{ev}_{\tau}}{\longrightarrow} E\to 0.
  \]
  Since $\Spec R_{\ker\tau}= \Spec R \times_{S_G} S_{\ker\tau} $ and $S_{\ker\tau}=S\times_E S_G$, there is a pullback square
  \[
  \begin{tikzcd}
    \Spec R_{\ker\tau} \ar[r,"i_{\ker\tau}"]\ar[d] & \Spec R \ar[d, "f_\tau"]\\
    S \ar[r,"0_E"] & E.
  \end{tikzcd}
  \]
  Here $f_\tau$ is the composition of $\operatorname{ev}_\tau$ and the inclusion $\Spec R\to S_G$, so $f_\tau$ is flat. The zero section of $E$ gives an exact sequence of sheaves over $E$
  \[
  0\to \mathcal{O}_E(-[0]) \to \mathcal{O}_E \to (0_E)^*\mathcal{O}_S \to 0.
  \]
  Pulling back this exact sequence along $f_\tau$ gives the exact sequence 
  \[
  0\to f_\tau^*\mathcal{O}_E(-[0]) \to \mathcal{O}_{\Spec R} \to (i_{\ker\tau})^*\mathcal{O}_{\Spec R_{\ker\tau}} \to 0
  \]
  of sheaves over $\Spec R$. Here $f_\tau^*\mathcal{O}_E(-[0]) \to \mathcal{O}_{\Spec R}$ is exactly the Euler class $e_\tau$ on $\Spec R$ by our definition. Therefore this exact sequence gives an exact sequence of rings 
  \[
  0\longrightarrow R\overset{e_\tau}{\longrightarrow} R \longrightarrow R_{\ker\tau}\to 0,
  \]
  and $e_\tau$ is regular in $R$. 
  
  Consider the sequence $(e_{\tau_1}, e_{\tau_2},\dots, e_{\tau_r})$, since $\{\tau_i\}$ are linearly independent, $$R/(e_{\tau_1}, e_{\tau_2},\dots, e_{\tau_{r-1}})\cong R_{\ker \sum_{i=1}^{r-1}\tau_i}.$$ Since $e_{\tau_r}$ is regular in $R_{\ker \sum_{i=1}^{r-1}\tau_i}$, the sequence $(e_{\tau_1}, e_{\tau_2},\dots, e_{\tau_r})$ is regular over $R$.

  (ii) With the same argument of \Cref{lem:adapted-group}, we can choose an inclusion $H\hookrightarrow \widetilde{H}$ and a surjective non-torsion character $u:\widetilde{H} \to \bbT$ such that $\ker(u)=K$ and $\ker(pu)=H$. By enlarging $G$ if necessary and passing to a suitable fpqc cover, we may assume that $\widetilde H$ is a subgroup of $G$, so that $R_{\widetilde H}$ is well-defined. Then $R_H\cong R_{\widetilde{H}}/(e_{pu})$, $e_u|e_{pu}\in R_{\widetilde{H}}$, and the projection of $e_u$ in $R_{\widetilde{H}}/(e_{pu})$ is $e_{H/K}$. Since $e_{pu}$ is a nonzero divisor by (i), it follows from \Cref{lem:compact-exact-regular} that $(e_{H/K},d_{H/K})$ is an exact pair in $R_H$.
\end{proof}

\begin{proposition}\label{prop:elliptic formal group is exact}
  For any elliptic $E/S$ represented by a flat map $S\to \mathcal{M}_{ell}$, the $G$-equivariant formal group 
  \[
  \phi_G: S_G \to \MFG^G
  \]
  is flat.
\end{proposition}
\begin{proof}
  Let $\{U_i\}_{i\in I}$ be a fpqc cover of $S_G$ such that $U_i=\Spec R_i$, and $\phi_G|_{U_i}$ admits a coordinate. It suffices to show that for every $i$, $\Spec R_i$ is flat over $\MFG^G$. By \Cref{thm:compact-LEFT}, we need to show that $R_i$ satisfies the condition $\mathrm{(RE)}+\mathrm{(EU)}+\mathrm{(LE)}$. 
  
  The condition $\mathrm{(RE)}$ and $\mathrm{(EU)}$ follow from \Cref{lem:condition Re and Eu for elliptic}. For any closed subgroup $H$, $S\to \mathcal{M}_{ell}$ is flat, $S_H$ is flat over $S$, and $\Spec (R_i)_H$ is flat over $S_H$. Thus $\Spec (R_i)_H$ is flat over $\MFG$, and the condition $\mathrm{(LE)}$ follows from the Landweber exactness of $(R_i)_H$ and the flatness of $\Phi^H R$ over $R_H$ for every closed subgroup $H$. 
\end{proof}

As a result, the equivariant elliptic cohomology theories defined in \cite[Definition~7.3]{strickland2011multicurves} and \cite[Section~3]{hu2021equivariant} are Landweber exact. 

\subsection{Equivariant formal groups from elliptic cohomology theories}
We conclude this section by studying the equivariant cohomology theory constructed by 
Gepner-Meier\cite[Construction~6.4]{gepner2023equivariant}. 
In what follows, we use sans-serif letters to denote spectral objects.  

For a spectral scheme $\mathsf S$, 
let $\mathsf S_0=(S,\pi_0\mathcal{O}_S)$ be the underlying classical scheme of $\mathsf S$. For any oriented spectral elliptic curve $\mathsf E/\mathsf S$ with underlying classical elliptic curve $E/S=\mathsf E_0/\mathsf S_0$, Gepner-Meier \cite[Construction~6.4]{gepner2023equivariant} defines a $G$-equivariant cohomology theory as a colimit-preserving functor
\[
\mathcal Ell_{G}: \mathrm{Space}^G \to \mathrm{QCoh}(\mathsf{ELL}(\mathbf BG))^{op},
\]
which takes values in the category of quasi-coherent sheaves over $\mathsf{ELL}(\mathbf BG)=\hom_{\mathrm{Ab}(\mathrm{Sh}(\mathsf S))}(G^*, \mathsf E)$. Note that when $G=\mathbb T^n$ is a torus, $\mathcal Ell_G$ is symmetric monoidal \cite[Theorem~14.19(3)]{Gepner_Linskens_Pol_2026}. 
Let $\mathsf{E}_G=\mathsf{ELL}(\mathbf{B}G)\times_{\mathsf S} \mathsf E$ be the oriented spectral elliptic curve over $\mathsf{ELL}(\mathbf{B} G)$, then the underlying spectral of 
\[
\mathsf p_G:\mathsf E_G \to \mathsf{ELL}(\mathbf{B} G)
\]
is exactly the elliptic curve $p_G:E_G\to S_G$ we constructed above. Let $\mathsf s_\chi: \mathsf{ELL}(\mathbf{B} G)\to \mathsf E_G$ be the section defined by the evaluation of $\chi\in G^*$. We also write $p_V: D_V\to S_G$ for the divisor of $E_G$ associated to a finite $G$-representation $V$. 

\begin{lemma}\label{lem: line bundle over spectral E_G}
  For any $\chi\in G^*$, let $\mathscr{L}_{\chi}$ be the fiber of $\mathcal{O}_{\mathsf E_G}\to (\mathsf s_\chi)_*\mathcal{O}_{\mathsf{ELL}(\mathbf{B}G)}$. Then 
  \[
  \pi_k \mathscr{L}_\chi \cong \mathcal{O}_{E_G}(-[s_\chi])\otimes_{\mathcal{O}_{E_G}}(p_G)^*\pi_k\mathcal{O}_{\mathsf{ELL}(\mathsf BG)}.
  \]
\end{lemma}
\begin{proof}
  Note that $\mathsf{ELL}(\mathbf{B} G\times S^1)=\mathsf{ELL}(\mathbf{B} G)\times_{\mathsf S} \mathsf E=\mathsf E_G$. Let $\rho_\chi: G\times S^1 \to S^1$ be the group homomorphism defined by $\chi$. By \cite[lemma 9.1]{gepner2023equivariant}, $\widetilde{\mathcal{E}ll_{S^1}}(S^\tau)\cong \mathcal{O}_{\mathsf E}(-e)$ where $e:\mathsf S\to \mathsf E$ is the zero section of $\mathsf E$. Base change along $\mathbf{B}\rho_\chi$, it follows from \cite[Theorem~14.19 (2)]{Gepner_Linskens_Pol_2026} that 
  \[
  \widetilde{\mathcal{E}ll_{G\times S^1}}(S^\chi)\cong \rho_{\chi}^*\mathcal{O}_{\mathsf E}(-e).
  \]
  The map $\mathsf{ELL}(\mathbf{B} G\times S^1)\to \mathsf{ELL}(\mathbf{B} S^1)$ induced by $\rho_{\chi}$ is
  \[
  \mathsf E_G\to \mathsf E, \qquad (f,x)\mapsto f(\chi)-x,
  \]
  whose zero fiber is given by $\mathsf s_{\chi}$. Thus $\rho_{\chi}^*\mathcal{O}_{\mathsf E}(-e) \cong \mathscr{L}_{\chi}$.
  Applying $\widetilde{\mathcal{E}ll_{G\times S^1}}$ to the cofiber sequence $S(\chi)_+\to S^0\to S^\chi$
  results in a fiber sequence of sheaves of $E_\infty$-rings
  \[
  \mathscr{L}_{\chi} \to \mathcal{O}_{\mathsf E_G} \to \widetilde{\mathcal{E}ll_{G\times S^1}}(S^\chi).
  \]
  The flatness of $\mathsf p_G$ implies that $\pi_k \mathcal O_{\mathsf E_G} \cong (p_G)^* \pi_k \mathcal{O}_{\mathsf{ELL}(\mathbf BG)}$. 
  Since $\mathsf s_\chi$ is affine, we have $$\pi_k (\mathsf s_\chi)_* \mathcal{O}_{\mathsf{ELL}(\mathsf BG)}\cong (s_\chi)_* \pi_k \mathcal{O}_{\mathsf{ELL}(\mathsf BG)}$$ by \cite[Lemma 2.8]{meier2022topological}, here $s_\chi$ is the underlying classical map of $\mathsf s_\chi$. Consider the cofiber sequence  Since $(s_\chi)_* \mathcal{O}_{S_G}$ is flat over $\mathcal{O}_{S_G}$, tensoring  the exact sequence of sheaves over classical elliptic curve $E_G/S_G$
  \[
  0\to \mathcal{O}_{E_G}(-[s_\chi])\to \mathcal{O}_{E_G}\to (s_\chi)_*\mathcal{O}_{S_G}\to 0
  \]
  with $(p_G)^*\pi_k\mathcal{O}_{\mathsf{ELL}(\mathsf BG)}$ is also exact, so $\mathcal{O}_{\mathsf E_G}\to \widetilde{\mathcal{E}ll_{G\times S^1}}(S^\chi)$ is surjective on homotopy groups in all degrees. Therefore, $\mathscr{L}_{\chi} \to \mathcal{O}_{\mathsf E_G}$ is injective on homotopy groups in all degree and 
  \[
  \pi_k \mathscr{L}_\chi \cong \mathcal{O}_{E_G}(-[s_\chi])\otimes_{\mathcal{O}_{E_G}}(p_G)^*\pi_k\mathcal{O}_{\mathsf{ELL}(\mathsf BG)}.
  \]
\end{proof}

\begin{lemma}\label{lem:complex projective space CPV}
  Let $V$ be a finite $G$-representation. Then 
  \[
  \pi_0(\mathcal{E}ll_{G}(\mathbb CP^V)) \cong (p_V)_*\mathcal{O}_{D_V}.
  \]
\end{lemma}
\begin{proof}
  Let $\tau\in (S^1)^*$ be the natural representation of $S^1$. For any finite $V=\bigoplus\limits_{i=1}^r \chi_i$, let $W=V\otimes \tau^{-1}$ be the $G\times S^1$-representation. Then $S(W)$ is a $G\times S^1$-space with a free $S^1$-action, and $\mathbb CP^V \simeq S(W)/S^1$. 
  It folows from \cite[Theorem~14.19 (1)]{Gepner_Linskens_Pol_2026} that
  \[
  \mathcal{E}ll_{G}(\mathbb CP^V)\simeq \alpha_* \mathcal{E}ll_{G\times S^1}(S(W)),
  \]
  where $\alpha:\mathsf{ELL}(\mathbf{B}G\times S^1) \to \mathsf{ELL}(\mathbf{B}G)$ is the morphism induces by the projection $G\times S^1\to G$. Note that $\mathsf{ELL}(\mathbf{B} G\times S^1)=\mathsf{ELL}(\mathbf{B} G)\times_{\mathsf S} \mathsf E=\mathsf E_G$, $\alpha$ is exactly the spectral elliptic curve $\mathsf p_G: \mathsf E_G \to \mathsf{ELL}(\mathbf{B}G)$.

  Let $\rho_W$ be the group homomorphism defined by $W$:
  \[
  \rho_W: G\times S^1 \to \mathbb T^r,\qquad (g,t) \mapsto (\chi_1(g)t^{-1},\dots, \chi_r(g)t^{-1}) .
  \]
  The K\"{u}nneth formula for tori \cite[Theorem~14.19 (3)]{Gepner_Linskens_Pol_2026} implies 
  $$\widetilde{\mathcal{E}ll_{\mathbb T^r}}(S^{\bigoplus_{i=1}^r\tau_i})\cong \bigotimes_{i=1}^r \pi_i^*\mathcal{O}_{\mathsf E}(-e),$$
  where $\pi_i:\mathbb T^r\to S^1$ is the projection to the $i$-th coordinate, and the $T^r$-representation $\tau_i$ is given by $\pi_i$. Base change along $\mathbf{B}\rho_W$, it follows from \cite[Theorem~14.19 (2)]{Gepner_Linskens_Pol_2026} that 
  \[
  \widetilde{\mathcal{E}ll_{G\times S^1}}(S^W)\cong \bigotimes_{i=1}^r \rho_W^* \pi_i^*\mathcal{O}_{\mathsf E}(-e).
  \]
  The map $\mathsf{ELL}(\mathbf{B} G\times S^1)\to \mathsf{ELL}(\mathbf{B} S^1)$ induced by the $i$-th component of $\rho_W$ is
  \[
  \mathsf E_G\to \mathsf E, \qquad (f,x)\mapsto f(\chi_i)-x,
  \]
  whose zero fiber is given by $\mathsf s_{\chi_i}$. Thus $\rho_W^* \pi_i^*\mathcal{O}_{\mathsf E}(-e) \cong \mathscr{L}_{\chi_i}$ and 
  $$\widetilde{\mathcal{E}ll_{G\times S^1}}(S^W)\cong \bigotimes_{i=1}^r\mathscr{L}_{\chi_i}.$$
  Applying $\widetilde{\mathcal{E}ll_{G\times S^1}}$ to the cofiber sequence 
  $S(W)_+ \to S^0 \to S^{W}$
  results in a fiber sequence of sheaves of $E_\infty$-rings
  \[
  \bigotimes_{i=1}^r\mathscr{L}_{\chi_i} \to \mathcal{O}_{\mathsf E_G}\overset{\eta}{\to} \widetilde{\mathcal{E}ll_{G\times S^1}}(S(W)_+).
  \]
  Note that \Cref{lem: line bundle over spectral E_G} implies that $\mathscr{L}_{\chi_i}$ is flat of rank $1$ over $\mathsf E_G$ for any $\chi_i$. By \cite[Proposition 7.2.2.13]{lurie2017ha}
  \[
  \begin{aligned}
    \pi_k \bigotimes_{i=1}^r\mathscr{L}_{\chi_i} & \cong \big(\bigotimes_{i=1}^r\mathcal{O}_{E_G}(-[s_{\chi_i}])\big)\otimes_{\mathcal{O}_{E_G}}(p_G)^*\pi_k\mathcal{O}_{\mathsf{ELL}(\mathsf BG)}\\
    &\cong \mathcal{O}_{E_G}(-D_V)\otimes_{\mathcal{O}_{E_G}}(p_G)^*\pi_k\mathcal{O}_{\mathsf{ELL}(\mathsf BG)}.
  \end{aligned}
  \]
  This imples that $\bigotimes_{i=1}^r\mathscr{L}_{\chi_i} \to \mathcal{O}_{\mathsf E_G}$ is injective on homotopy groups in all degrees. So $\eta$ is surjective on homotopy groups in all degrees and 
  \[
  \pi_k\widetilde{\mathcal{E}ll_{G\times S^1}}(S(W)_+)\cong (i_V)_* (p_V)^* \pi_k \mathcal{O}_{\mathsf{ELL}(\mathbf BG)},
  \]
  where $i_V$ is the inclusion of the divisor $D_V\to E_G$. In particular, when $k=0$, 
  \[      
  \pi_0\widetilde{\mathcal{E}ll_{G\times S^1}}(S(W)_+)\cong (i_V)_*\mathcal{O}_{D_V}.
  \]
  We use the relative descent spectral sequence 
  \[
  E_2^{s,t}=R^s((p_0)_*)\pi_t\widetilde{\mathcal{E}ll_{G\times S^1}}(S(W)_+)\Longrightarrow\pi_{t-s}(\mathsf p_G)_*\widetilde{\mathcal{E}ll_{G\times S^1}}(S(W)_+)
  \]
  to compute $\pi_*\mathcal{E}ll_{G}(\mathbb CP^V)$, here $R^s((p_0)_*)$ is the derived functor of the classical pushforward functor $(p_0)_*:\mathrm{Shv}(E_G)\to \mathrm{Shv}(S_G)$, see also in \cite[Theorem 10.1]{gepner2023equivariant}. Since $(p_V)^* \pi_k \mathcal{O}_{\mathsf{ELL}(\mathbf BG)}$ is quasi-coherent and $p_V,i_V$ are affine, \cite[Lemma 30.2.3]{stacks-project} implies that 
  \[
  \begin{aligned}
    R^s((p_0)_*)\pi_t\widetilde{\mathcal{E}ll_{G\times S^1}}(S(W)_+) & \cong R^s((p_0)_*) R^0((i_V)_*) (p_V)^* \pi_t \mathcal{O}_{\mathsf{ELL}(\mathbf BG)} \\
    & \cong R^s((p_V)_*) (p_V)^* \pi_t \mathcal{O}_{\mathsf{ELL}(\mathbf BG)}
  \end{aligned}
  \]
  which vanishes for all $s>0$. Here the second isomorphism follows from the Relative Leray spectral sequence \cite[Lemma 20.13.8]{stacks-project}. Thus $E_2$ is concentrated in the zero-line and 
  \[
  \pi_{t}(\mathsf p_G)_*\widetilde{\mathcal{E}ll_{G\times S^1}}(S(W)_+)\cong (p_V)_*p_V^* \pi_t \mathcal{O}_{\mathsf{ELL}(\mathbf BG)}.
  \]
  When $t=0$, it gives 
  \[
  \pi_0 \mathcal{E}ll_{G}(\mathbb CP^V) \cong (p_V)_* \mathcal{O}_{D_V}.
  \]
\end{proof}

For any finite $G$-representation $V=\bigoplus_{i=1}^r \chi_r$ and the corresponding $G\times S^1$-representation $W=V\otimes \tau^{-1}$, let $\mathsf D_V=\Spec_{\mathsf E_G} (\widetilde{\mathcal{E}ll_{G\times S^1}}S(W)_+)$. Here $\Spec_{\mathsf E_G}$ is the functor defined in \cite[Lemma 8.1]{gepner2023equivariant}, which is the inverse of
\[
\mathrm{Aff}_{\mathsf E_G}^{op}\to CAlg(QCoh(\mathsf E_G)), \qquad (\mathsf U\overset{f}{\to}\mathsf \mathsf E_G) \mapsto f_*\mathcal{O}_{\mathsf U}.
\]
It follows from the computation in \Cref{lem:complex projective space CPV} that $\mathsf i_V:\mathsf D_V\to \mathsf E_G$ is a closed immersion with ideal sheaf $\mathcal{O}_{\mathsf E_G}(-D_V)\cong \bigotimes_{i=1}^r\mathscr{L}_{\chi_i}$. So $\mathsf i_V:\mathsf D_V\to \mathsf E_G$ is an effective Cartier divisor in the sense of \cite[Definition~2.14]{davies2026equivariantaffinenessellipticcohomology}. Moreover, the underlying scheme of $\mathsf i_V$ is the divisor $i_V:D_V\to E_G$, and the pushforward of the structure sheaf along 
\[
\mathsf p_V: \mathsf D_V\overset{\mathsf i_V}{\longrightarrow} \mathsf E_G \overset{\mathsf p_G}{\longrightarrow} \mathsf{ELL}(\mathbf BG)
\]
gives $(\mathsf p_V)_* \mathcal{O}_{\mathsf D_V}\simeq \mathcal{E}ll_{G}(\mathbb CP^V)$. Since $\mathsf D_V$ is affine over $\mathsf{ELL}(\mathbf BG)$, this equivalence implies that $\mathsf D_V \simeq \Spec_{\mathsf{ELL}(\mathbf BG)}(\mathcal{E}ll_{G}(\mathbb CP^V))$, which implies that the construction of $\mathsf D_V$ is independent of the chosen decomposition of $V$.

For finite $G$-representations $V_1,V_2$, let $\rho_i:G \to \mathbb T^{\dim V_i}$ be the group homomorphism defined by $V_i$. Since $\mathcal Ell_{T}$ is symmetric monoidal when $T$ is a torus, there is an equivalence 
\[
\mathcal Ell_G(\mathbb CP(V_1))
\otimes_{\mathcal O_{\mathsf{ELL}(\mathbf BG)}}
\mathcal Ell_G(\mathbb CP(V_2))
\longrightarrow
\mathcal Ell_G\bigl(\mathbb CP(V_1)\times\mathbb CP(V_2)\bigr)
\]
which follows from the equivalence for $\mathbb T^{\dim V_1+\dim V_2}$ and a base change via $(\rho_1, \rho_2): G\to \mathbb T^{\dim V_1+\dim V_2}$. 
Then the embedding
\[
l_{V_1,V_2}:\mathbb CP(V_1)\times\mathbb CP(V_2)\longrightarrow\mathbb CP(V_1\otimes V_2),
\qquad (L,M)\longmapsto L\otimes M
\]
gives a morphism 
$\mathcal Ell_G(\mathbb CP(V_1\otimes V_2))\to \mathcal Ell_G(\mathbb CP(V_1))
\otimes_{\mathcal O_{\mathsf{ELL}(\mathbf BG)}}\mathcal Ell_G(\mathbb CP(V_2))$, which duals to
\[
\sigma_{V_1,V_2}\colon
\mathsf D_V\times_{\mathsf S_G}\mathsf D_W
\longrightarrow \mathsf D_{V\otimes W}.
\]

\begin{proposition}\label{prop:compactible of elliptic formal groups}
  Let $\mathcal{U}_G$ be a complete universe of $G$. Then 
  \[
  \mathcal{E}ll_{G}(\mathbb CP(\mathcal{U}_G)) \cong \lim_{\leftarrow} \ (\mathsf p_V)_*\mathcal{O}_{\mathsf D_V},
  \]
  where the inverse limit is taken over the finite dimensional $G$-representations $V\subset \mathcal U_G$. In particular, 
  $$\operatorname{colim}_V \big(\Spec_{\mathsf{ELL}(\mathbf BG)}(\mathcal{E}ll_{G}(\mathbb CP^V))\big)_0$$ 
  gives the $G$-equivariant formal group $\widehat{E_G}/S_G$ defined in \Cref{lem:G-formal group associated to elliptic curves}.
\end{proposition}
\begin{proof}
  For any finite $G$-representation $V$, let $W=V\otimes \tau^{-1}$ for the natural $S^1$-representation. It follows from the proof of \Cref{lem:complex projective space CPV} that 
  \[
  \mathcal{E}ll_{G}(\mathbb CP^V)\simeq \alpha_* \mathcal{E}ll_{G\times S^1}(S(W)),
  \]
  and we compute $\mathcal{E}ll_{G\times S^1}(S(W))$ via the fiber sequence
  \[
  \widetilde{\mathcal{E}ll_{G\times S^1}}(S^W) \to \mathcal{O}_{\mathsf E_G}\overset{\eta}{\to} \widetilde{\mathcal{E}ll_{G\times S^1}}(S(W)_+).
  \]
  Since the fiber sequence induces a short exact sequence on homotopy groups in every degree, the naturality of the equivariant elliptic cohomology implies that the isomorphism $\mathcal{E}ll_{G}(\mathbb CP^V)\cong (p_V)_*p_V^* \pi_t \mathcal{O}_{\mathsf{ELL}(\mathbf BG)}$ is compatible with any $G$-lenear embedding between $G$-representations. So $$\pi_k \mathcal{E}ll_{G}(\mathbb CP(\mathcal U_G))\cong \lim_{\leftarrow} \ (p_V)_*p_V^* \pi_k \mathcal{O}_{\mathsf{ELL}(\mathbf BG)}.$$

  Taking $k=0$, we have $\pi_0 \mathcal{E}ll_{G}(\mathbb CP(\mathcal U_G))\cong \lim\limits_{\leftarrow} \ (p_V)_*\mathcal{O}_{D_V}.$ For any finite $G$-representation $V_1,V_2$, the underlying morphism of 
  $\sigma_{V_1,V_2}\colon \mathsf D_V\times_{\mathsf S_G}\mathsf D_W \to \mathsf D_{V\otimes W}$ induces a group structure of the underlying formal scheme of $\operatorname{colim}_V\mathsf D_V$. Since the underlying scheme $(\mathsf D_V)_0\cong D_V$, it suffices to show that $\sigma_{V_1,V_2}$ is compatible with the addition map of $\mathsf E_G$.

  Let $W_i=V_i\otimes \tau^{-1}$, consider the map 
  \[
  L_{V_1,V_2}:S(W_1)\times S(W_2) \to S(V_1\otimes V_2 \otimes \tau^{-1}),\quad  (v_1,v_2)\mapsto v_1v_2.
  \]
  Here $G\times \mathbb T^2$ acts on the left by $(g,t,u)\cdot (v_1,v_2) =(t^{-1}g\cdot v_1, u^{-1}g\cdot v_2)$, which is maps to $(tu)^{-1}g \cdot v_1v_2$. 
  Then $L_{V_1,V_2}$ is equivariant with respect to 
  \[
  \beta: G\times \mathbb T^2 \to G \times \mathbb T, \qquad (g,t,u)\mapsto (g,tu).
  \]
  Note that $\mathbb T^2$ and $\mathbb T$ acts freely on $S(W_1)\times S(W_2)$ and $S(V_1\otimes V_2 \otimes \tau^{-1})$ respectively. Passing to the quotients by the free actions of $\mathbb T^2$ and $\mathbb T$ gives the commutative square
  \[
  \begin{tikzcd}
    S(W_1)\times S(W_2) \arrow[d]\arrow[r,"L_{V_1,V_2}"] & S(V_1\otimes V_2\otimes\tau^{-1})\arrow[d]\\
    \mathbb CP^{V_1}\times \mathbb CP^{V_2} \arrow[r,"l_{V_1,V_2}"] & \mathbb CP^{V_1\otimes V_2}.
  \end{tikzcd}
  \]
  Let $\alpha:G\times \mathbb T\to G$ be the projection to $G$, then this diagram of spaces gives a commutative diagram on cohomology
  \[
  \begin{tikzcd}
    \mathcal Ell_{G}(\mathbb CP^{V_1\otimes V_2})\arrow[r,"Q_\alpha"]\arrow[dd,"(l_{V_1,V_2})^*"]
    & \alpha_*\mathcal Ell_{G\times \mathbb T}(S(V_1\otimes V_2\otimes\tau^{-1})) \arrow[d,"Q_\beta"]\\ 
    & \alpha_*\beta_* \mathcal Ell_{G\times \mathbb T^2} (S(V_1\otimes V_2\otimes \tau^{-1}))\arrow[d,"(L_{V_1,V_2})^*"]\\
    \mathcal Ell_{G}(\mathbb CP^{V_1}\times \mathbb CP^{V_2} ) \arrow[r,"Q_{\alpha\circ\beta}"] 
    & \alpha_*\beta_*\mathcal Ell_{G\times \mathbb T^2}(S(W_1)\times S(W_2)),
  \end{tikzcd}
  \]
  where $Q_\alpha$ is the base change map associated to $\alpha$. Since $\mathsf{ELL}(\mathbf B G\times \mathbb T^2) \simeq \mathsf E_G\times_{\mathsf{ELL}(\mathbf BG)}\mathsf E_G$, the map $\mathbf B\beta : \mathsf E_G\times_{\mathsf{ELL}(\mathbf BG)}\mathsf E_G \to \mathsf E_G$ is the addition map of $\mathsf E_G$, and  $\beta_*$ is the pushforward functor along $\mathbf B\beta$. By \Cref{lem:complex projective space CPV}, 
  $$\mathcal Ell_{G\times \mathbb T}(S(V_1\otimes V_2\otimes\tau^{-1}))\simeq(\mathsf i_{V_1\otimes V_2})_* \mathcal O_{\mathsf D_{V_1\otimes V_2}},\quad \mathcal Ell_{G\times \mathbb T^2} (S(W_1)\times S(W_2))\simeq(\mathsf i_{V_1}\times \mathsf i_{V_2})_* \mathcal O_{\mathsf D_{V_1}\times \mathsf D_{V_2}},$$
  correspond to the divisor $\mathsf i_{V_1\otimes V_2}:\mathsf D_{V_1\otimes V_2}\to E_G$ and $\mathsf i_{V_1}\times \mathsf i_{V_2}: \mathsf D_{V_1}\times \mathsf D_{V_2}\to \mathsf E_G\times_{\mathsf{ELL}(\mathbf BG)}\mathsf E_G$, respectively.
  Then the map in the right vertical of the diagram 
  \[
  (L_{V_1,V_2})^*\circ Q_\beta\colon \mathcal Ell_{G\times \mathbb T}(S(V_1\otimes V_2\otimes\tau^{-1})) \longrightarrow \beta_*\mathcal Ell_{G\times \mathbb T^2}(S(W_1)\times S(W_2)),
  \]
  is adjoint to a morphism of sheaves over $E_G\times_{\mathsf{ELL}(\mathbf BG)}\mathsf E_G$
  \[
  \beta^*\mathcal Ell_{G\times \mathbb T}(S(V_1\otimes V_2\otimes\tau^{-1})) \longrightarrow \mathcal Ell_{G\times \mathbb T^2}(S(W_1)\times S(W_2)),
  \]
  which corresponds to $\mathsf D_{V_1}\times \mathsf D_{V_2} \to (\mathsf E_G\times_{\mathsf{ELL}(\mathbf BG)}\mathsf E_G)\times_{\mathsf E_G} \mathsf D_{V_1\otimes V_2}$. Composing with $\mathbf B\beta$ gives the following commutative diagram 
  \[
  \begin{tikzcd}
    \mathsf D_{V_1}\times \mathsf D_{V_2} \arrow[d,"\mathsf i_{V_1}\times \mathsf i_{V_2}"]\arrow[r,"\sigma_{V_1,V_2}"] & \mathsf D_{V_1\otimes V_2}\arrow[d, "\mathsf i_{V_1\otimes V_2}"]\\
    \mathsf E_G\times_{\mathsf{ELL}(\mathbf BG)}\mathsf E_G \arrow[r,"\mathbf B\beta"] & \mathsf E_G.
  \end{tikzcd}
  \]
  So $\sigma_{V_1,V_2}$ is compatible with the addition map of $\mathsf E_G$. Then passing to the underlying classical schemes and taking colimits over $V$ gives the $G$-equivariant formal group $\widehat{E}_G/S_G$.
\end{proof}

Following the idea of \cite[Lemma~2.1]{Chua2022} and \cite[Remark~11.14]{Gepner_Linskens_Pol_2026}, the functor $\mathcal Ell_G$ can be extended to a colimit-preserving functor of $\Sp^G$. For any \'{e}tale chart $\mathsf U$ over $\mathsf{ELL}(\mathbf BG)$, the composition of $\mathcal Ell_G$ with the functor $\Gamma_{\mathsf U}$ taking sections on $\mathsf U$ is a colimit-preserving functor
\[
\Sp^{G,fin} \overset{\mathcal Ell_G}{\longrightarrow} \mathrm{QCoh}(\mathsf{ELL}(\mathbf BG)) \overset{\Gamma_{\mathsf U}}{\longrightarrow} \Sp^{op},
\]
which is represented by a genuine $G$-spectrum $R_{G,\mathsf U}$. Here we leave the dependency on $\mathsf E$ implicit. Note that if we take $\mathsf U=\mathsf{ELL}(\mathbf BG)$, then $R_G:=R_{G,\mathsf{ELL}(\mathbf BG)}$ is the equivariant elliptic spectrum constructed by Gepner-Meier \cite[Construction~9.3]{gepner2023equivariant}. 

In particular, if $E/S$ is represented by an \'{e}tale morphism $S\to \mathcal M_{ell}$. 
Let $\mathsf S=(S, \mathcal O^{top}|_{S})$, and let $\mathsf E \to \mathsf S$ be the oriented elliptic curve determined by $$\mathsf S\to \mathsf M_{ell}^{or},$$ where $\mathsf M_{ell}^{or}$ is the non-connective spectral Deligne-Mumford stack of oriented elliptic curves \cite[Proposition~7.2.10]{lurie2018elliptic}, and $\mathcal O^{top}$ is the structure sheaf of $\mathsf M_{ell}^{or}$. Then $\mathsf E/\mathsf S$ gives a spectral lifting of $E/S$. In this case, when $\mathsf U$ is affine, then \Cref{prop:compactible of elliptic formal groups} implies that the $G$-formal group associated to $R_{G,\mathsf U}$ is given by the pullback 
\[
\begin{tikzcd}
  \mathsf U_0\times_{S_G}\widehat{E_G}\arrow[r]\arrow[d] & \widehat{E_G}\arrow[d]\\
  \mathsf U_0 \arrow[r] & S_G.
\end{tikzcd}
\]
In particular, if $\mathsf U_0\times_{S_G}\widehat{E_G}/\mathsf U_0$ has a coordinate, then this gives a Landweber exact $G$-formal group law and $R_{G,\mathsf U}$ is complex oriented. The affineness of $\mathsf{ELL}(\mathbf BH)\to \mathsf{ELL}(\mathbf BG)$ implies that $\pi_{2k}^{H}R$ is the homotopy group of the global section of $\mathcal{O}_{U\times_{\mathsf{ELL}(\mathbf BG)}\mathsf{ELL}(\mathbf BH)}$. Since $R_{G,\mathsf U}$ is even-periodic, $\pi_{2k}^{G}R = \pi_{2k} \Gamma\mathcal O_{\mathsf U}\cong \Gamma\mathcal{O}_{\mathsf U_0}$, it follows from \Cref{cor:compact representation} and \Cref{lem:change of groups} that $R_{G,\mathsf U}$ is equivariant Landweber exact.

\begin{remark}
  When $G=\mathbb T$, by \cite[Theorem~10.1]{gepner2023equivariant}, restriction along $\mathbb T\to e$ gives a splitting 
  $$(R_{\mathbb T})^{\mathbb T}\simeq R_{\mathbb T} \oplus \Sigma R_{\mathbb T}.$$ Since the underlying spectrum $(R_{\mathbb T})^e$  is even-periodic, the restrictions map on the even homotopy groups $\pi_{2k}^{\mathbb T}(R_{\mathbb T})\cong \pi_{2k}^e(R_{\mathbb T})$ is an isomorphism. If $R_{\mathbb T}$ admits an orientation, the Euler class $e_\tau$ of the identity character should be zero since it restricts to the Euler class of the trivial complex line, which is $0\in \pi_*^e(R_{\mathbb T})$. However, $e_\tau$ should be a non-zero divisor in $\pi_*^{S^1}(R_{\mathbb T})$, this contradiction shows that $R_{\mathbb T}$ cannot be complex oriented.

  Analogously, the $\mathbb T$-formal group from $E/S$ cannot admit a coordinate. Indeed, the elliptic curve $E_{\mathbb T}\to S_{\mathbb T}$ is given by the projective $p_1:E\times_S E\to E$. 
  The section $\Phi_{\mathbb T}(\tau): E\to E\times_S E$ is the diagonal map of $E$, which leads to a section $\Delta: E\to \widehat{E_{\mathbb T}}$. A coordinate leads to an isomorphism of sheaves $\mathcal O_{\widehat{E_{\mathbb T}}}\cong \mathcal O_{\widehat{E_{\mathbb T}}}(-[\epsilon])$ over $\widehat{E_{\mathbb T}}$. Base change along $\Delta$ gives an isomorphism 
  \[
  \mathcal{O}_E \cong \mathcal{O}_E(-[0]),
  \]
  which is impossible. So $\widehat{E_{\mathbb T}}\to S_{\mathbb T}$ can't admit a coordinate.
\end{remark}

\section*{Declaration on the Use of AI Tools}
During the preparation of this manuscript, the authors used ChatGPT (OpenAI) as an
auxiliary tool for literature searches and improving the English language. The research problems, central ideas, and main conceptual framework of this work were conceived by the authors.

\bibliographystyle{alpha}
\bibliography{homotopy.bib}
\end{document}